\documentclass[11pt,a4paper]{amsart}%
\usepackage{amsfonts,amsmath,amssymb}
\usepackage{hyperref}
\usepackage{amssymb,amsthm,amsxtra}
\usepackage[usenames]{color}
\usepackage{amscd}
\usepackage{amsthm}
\usepackage{amsfonts}
\usepackage{amssymb}
\usepackage{amsmath}
\usepackage{mathabx}
\usepackage{graphicx}
\usepackage{enumerate}
\usepackage[T1]{fontenc}%
\providecommand{\U}[1]{\protect\rule{.1in}{.1in}}
\let\oldtocsubsection=\tocsubsection
\let\oldtocsubsubsection=\tocsubsubsection
\renewcommand{\tocsubsection}[2]{\hspace{1em}\oldtocsubsection{#1}{#2}}
\renewcommand{\tocsubsubsection}[2]{\hspace{2em}\oldtocsubsubsection{#1}{#2}}

\theoremstyle{plain}
\newtheorem {theorem }{Theorem}
\newtheorem{lemma}{Lemma}[section]

\newtheorem{theorem}[lemma]{Theorem}
\newtheorem{corollary}[lemma]{Corollary}

\newtheorem {conjecture }{Conjecture}
\newtheorem{thm}[lemma]{Theorem}
\newtheorem{prop}[lemma]{Proposition}
\newtheorem{lem}[lemma]{Lemma}
\newtheorem{cor}[lemma]{Corollary}
\newtheorem{introtheorem}{Theorem}

\newtheorem{introthm}[introtheorem]{Theorem}

\newtheorem{introprop}[introtheorem]{Proposition}

\theoremstyle{definition}
\newtheorem{definition}[lemma]{Definition}
\newtheorem{defn}[lemma]{Definition}

\newtheorem{exm}[lemma]{Example}

\newtheorem{example}[lemma]{Example}

\newtheorem {example }{Example}
\theoremstyle{remark}
\newtheorem {remark }{Remark}

\newtheorem{remark}[lemma]{Remark}
 \theoremstyle{plain}

\newcommand{\ind}{\operatorname{ind}}

\newcommand{\Rep}{\operatorname{Rep}}
\newcommand{\Hom}{\operatorname{Hom}}

\newcommand{\oH}{\operatorname{H}}

\newcommand{\Ind}{\operatorname{Ind}}

\renewcommand{\Im}{\operatorname{Im}}
\newcommand{\Ker}{\operatorname{Ker}}

\newcommand{\N}{{\mathbb N}}

\newcommand{\Q}{{\mathbb Q}}
\newcommand{\R}{{\mathbb R}}

\newcommand{\C}{{\mathbb C}}

\newcommand{\Aut}{{\operatorname{Aut}}}

\newcommand{\End}{\operatorname{End}}
\newcommand{\Exp}{\operatorname{Exp}}

\newcommand{\Id}{\operatorname{Id}}

\newcommand{\bC}{{\mathbb C}}

\newcommand{\alp}{{\alpha}}

\newcommand{\Fre}{{Fr\'{e}chet\,}}

\newcommand{\cA}{{\mathcal{A}}}

\newcommand{\cF}{{\mathcal{F}}}

\newcommand{\fg}{{\mathfrak{g}}}
\newcommand{\fh}{{\mathfrak{h}}}
\newcommand{\fn}{{\mathfrak{n}}}

\newcommand{\fu}{{\mathfrak{u}}}
\newcommand{\fv}{{\mathfrak{v}}}

\newcommand{\fz}{{\mathfrak{z}}}

\newcommand{\z}{{\mathfrak{z}}}

\newcommand{\GL}{\operatorname{GL}}

\newcommand{\Sp}{\operatorname{Sp}}

\newcommand{\sll}{{\mathfrak{sl}}}

\newcommand{\Mp}{\widetilde{\Sp}}

\newcommand{\Sc}{{\mathcal S}}
\newcommand{\Lie}{\operatorname{Lie}}

\newcommand{\hot}{\widehat{\otimes}}

\newcommand{\cM}{\mathcal{M}}

\newcommand{\fp}{\mathfrak{p}}
\newcommand{\fl}{\mathfrak{l}}
\newcommand{\fm}{\mathfrak{m}}

\newcommand{\fr}{\mathfrak{r}}

\newcommand{\cW}{\mathcal{W}}

\newcommand{\Dima}[1]{{{#1}}}
\newcommand{\Raul}[1]{{{#1}}}
\newcommand{\SSc}[1]{{{#1}}}
\newcommand{\DimaA}[1]{{{#1}}}
\newcommand{\DimaB}[1]{{{#1}}}

\newcommand{\onto}{{\twoheadrightarrow}}

\newcommand{\ot}{\leftarrow}

\newcommand{\supp}{\operatorname{supp}}

\newcommand{\WO}{\operatorname{WO}}

\renewcommand{\sl}{\mathfrak{sl}}

\renewcommand{\mod}{\operatorname{Mod}}

\newcommand{\htimes}{\hat\otimes}
\newcommand{\cU}{\mathcal{U}}
\newcommand{\pro}{\mathrm{pro}}
\newcommand{\WS}{\operatorname{WS}}

\begin{document}

\begin{abstract} The classical Stone-von Neuman theorem relates the irreducible unitary representations of the Heisenberg group $H_n$ to non-trivial unitary characters of its center $Z$, and plays a crucial role in the construction of the oscillator representation for the metaplectic group.

In this paper we  extend  these ideas  to non-unitary and non-irreducible representations, thereby obtaining an equivalence of categories between certain representations of $Z$ and those of $H_n$. Our main result is a smooth equivalence, which involves the fundamental ideas of du Cloux on differentiable representations and smooth imprimitivity systems for Nash groups. We show how to extend the oscillator representation to the smooth setting and give an application to degenerate Whittaker models for representations of reductive groups.

We also include  an algebraic equivalence, which can be regarded as a generalization of Kashiwara's lemma from the theory of $D$-modules.

\end{abstract}
\title[Stone-von Neumann equivalence for smooth representations]{A Stone-von Neumann equivalence of categories for smooth representations of the Heisenberg group}
\author{Raul Gomez}
\address{Raul Gomez, UANL FCFM Av. Universidad, San Nicolas de los Garza, N.L., Mexico}
\email{raul.gomez.rgm@gmail.com}

\author{Dmitry Gourevitch}
\address{Dmitry Gourevitch,
Faculty of Mathematics and Computer Science,
Weizmann Institute of Science,
234 Herzl Street, Rehovot 7610001 Israel}
\email{dmitry.gourevitch@weizmann.ac.il}
\urladdr{\url{http://www.wisdom.weizmann.ac.il/~dimagur}}

\author{Siddhartha Sahi}
\address{Siddhartha Sahi, Department of Mathematics, Rutgers University, Hill Center -
Busch Campus, 110 Frelinghuysen Road Piscataway, NJ 08854-8019, USA}
\email{sahi@math.rugers.edu}
\date{\today}
\maketitle

\section{Introduction}

The Heisenberg group $H=H_n$ plays an important role in many areas of mathematics and physics. It is the extension of a $2n$-dimensional symplectic vector space $W=W_{2n}$ by a one dimensional center $Z$. A fundamental result, due to Stone and von Neumann, relates characters of $Z$ to representations of $H$. More precisely if $\chi$ is a non-trivial unitary character of $Z$ then there is, up to equivalence, a unique irreducible unitary representation $\pi=\pi_\chi$ of $H$ in which $Z$ acts by the character $\chi$. The symplectic group $\Sp$ acts on $H$ by automorphisms and the Stone-von-Neumann theorem plays a crucial role in the construction of the Segal-Shale-Weil metaplectic representation of the double cover $\widetilde{\Sp}$,  indeed of the Fourier-Jacobi group $\widetilde{\Sp}\ltimes H$. This in turn is very important in the theory of automorphic forms and the $\theta$-correspondence.

In this paper we give two generalizations of this result to non-unitary, and non-irreducible, representations, obtaining equivalences of categories between suitable representations of $Z$ and of $H$. The first, easier, generalization given in Theorem \ref{thmA}, involves algebraic categories and can be regarded as an extension of Kashiwara's equivalence from the theory of $D$-modules \cite{Kas}.

Our main result is the second generalization, stated in  Theorem \ref{thmB} below. It involves smooth representations and is much more delicate. Its proof occupies the first half of the paper and involves in a crucial way the fundamental ideas of du Cloux on differentiable representations and smooth imprimitivity systems for Nash groups. Our results allow us to construct a metaplectic representation for an arbitrary non-degenerate representation of the center $Z$.  In fact, in Theorem \ref{thm:FJ}
below we show that this construction defines an equivalence of categories between certain representations of $\Sp\times Z$ and $\widetilde{\Sp}\ltimes H$.

Our results were motivated by certain natural questions arising in the study of degenerate Whittaker models for representations of reductive groups, which are useful in the theory of automorphic forms.  We describe one such application in Theorem \ref{thm:Int_proWhit} below, which can be regarded as an ``independence of polarization'' result for certain spaces of equivariant functionals on a smooth representation. We should like to think that in view of the ubiquity of the Heisenberg group, the results of the present paper will prove useful elsewhere as well.

\subsection{Statement of Theorem \ref{thmA}}\label{subsec:A}
Let $L\subset W$ be a maximally isotropic subspace, then $L$ and $P=LZ$ are abelian subgroups of $H$. We write $\fh, \fl, \fz, \fp$ for the Lie algebras of $H,L,Z,P$ and $\cU(\fh)$ etc. for their enveloping algebras. We choose a generator $z$ for the one-dimensional Lie algebra $\fz$, and write $\cM(\fz)^{\times}$ for the category of $\fz$-modules on which $z$ acts invertibly -- evidently this category does not depend on the choice of $z$. Our algebraic result concerns the category $\cM_{\fl}(\fh)^{\times}$ consisting of $\fh$-modules satisfying the following conditions
\begin{enumerate}
\item $\fl$ acts local nilpotently
\item $z$ acts invertibly
\end{enumerate}

%
%
\SSc{
For $n=0$, the Heisenberg group reduces to $Z$, and $\cM_{\fl}(\fh)^{\times}$ reduces to $\cM(\fz)^{\times}$.  We show that the categories are equivalent for all $n$; for this we define a pair of functors as follows

\[ \label{FG}  F:  \cM_{\fl}(\fh)^{\times}\leftrightarrows \cM(\fz)^{\times}: G,\quad  F(M):=M^\fl,\; G(N):=\cU(\fh)\otimes_{\cU(\fp)}E(N), \]
where $M^\fl$ denotes the subspace of $\fl$-invariants of an $\fh$-module $M$, and $E(N)$ denotes the extension of a $\fz$-module to a module for $\fp=\fl+\fz$, trivial on $\fl$.

\begin{introtheorem}\label{thmA}
The functors $F,G$ are quasi-inverses; thus $\cM_{\fl}(\fh)^{\times}$ and $\cM(\fz)^{\times}$ are equivalent.
\end{introtheorem}

In fact, Theorem \ref{thmA} is closely related to the Kashiwara equivalence   in the theory of $D$-modules \cite{Kas}, and can be proved using a similar strategy. Since we do not use this result in the main text, we defer the proof to Appendix \ref{app:Kash}.
\subsection{Statement of Theorem B}

We now give the smooth version of Theorem \ref{thmA}. If $G$ is a Nash group, for example the Heisenberg group $H$ or one of its subgroups $P,L,Z$, then we write $\Rep^{\infty}(G)$  for the category of smooth, nuclear \Fre representations of moderate growth \DimaA{(see \S \ref{sec:Sc} below)}.
For two representations $\pi,\tau\in \Rep^{\infty}(G)$ we denote by $\pi\hot \tau$ the completed tensor product, \Dima{by $\pi_G$ the coinvariants, and by
$\pi\hot_G\tau:=(\pi\hot\tau)_{\Delta G}$ the coinvariants of $\pi\hot \tau$ under the diagonal action of $G$ (see Definition \ref{def:coinv} for more details).}

Let $z$ be a generator of the Lie algebra $\fz$ of $Z$ as before; we say a representation $(\rho,M)$ in $\Rep^{\infty}(Z)$ is \emph{non-singular} if there exists a representation $(\rho',M)$ in $\Rep^{\infty}(Z)$ such that
$$d\rho(z)d\rho'(z)=d\rho'(z)d\rho(z)=1.$$
We write $\Rep^{\infty}(Z)^\times\subset \Rep^{\infty}(Z)$ for the category of non-singular $Z$-representations, and $\Rep^{\infty}(H)^\times\subset \Rep^{\infty}(H)$ for the category of $H$-representations with non-singular restriction to $Z$.

Our main result is that $\Rep^{\infty}(Z)^\times$ is equivalent to $\Rep^{\infty}(H)^\times$. As a first step  in \S \ref{sec:nonsing} below we show that the category $\Rep^{\infty}(Z)^\times$ is equivalent to the category of smooth modules  over the algebra $\Sc(\widehat{Z}^{\times})$ of Schwartz functions on the space $\widehat{Z}^{\times}$ of non-trivial unitary characters of $Z$. By a smooth module we mean a module satisfying $\Sc(\widehat{Z}^{\times})M=M$ (see \S \ref{sec:alg} below).
The algebra $\Sc(\widehat{Z}^{\times})$ is a smooth module over itself,  and the corresponding representation structure, that we denote by $\varpi$, is given by
$$\varpi(z)\phi(\chi):=\chi(z)\phi(\chi).$$
Let $E(\varpi)$ denote the trivial extension of $\varpi$ to a module for $P=LZ$ and define
$$\Omega:=\ind_{P}^H E(\varpi),$$
where ``$\ind$'' is the Schwartz induction functor for Nash groups defined by du Cloux  \cite{dCl} (see \S\ref{subsec:ImpThm} below).
Then $\Omega$ belongs to $\Rep^{\infty}(H)^\times$, and by Proposition \ref{prop:LagIndep} below, it does not depend on the choice of the Lagrangian $L$ (up to isomorphism).

We now define a pair of functors
$$C:\Rep^{\infty}(H)^\times \leftrightarrows \Rep^{\infty}(Z)^\times:I
\quad C(\tau):= \tau\hot_H \overline{\Omega}, \,\,I(\rho):=  \rho\hot_Z \Omega
$$
where $\overline{\Omega}$ is the complex conjugation of $\Omega$.

\begin{introtheorem}[See \S \ref{sec:cat}]\label{thmB}
The functors $C,I$ are quasi-inverses; thus $\Rep^{\infty}(H)^\times$ and  $\Rep^{\infty}(Z)^\times$ are equivalent.
\end{introtheorem}

The letter $C$ stands for coinvariants, and $I$ for induction. This is since for any choice of a Langrangian $L$, $C(\tau)$ is isomorphic to the coinvariants $\tau_L$ and $I(\rho)\cong \ind_P^HE(\rho)$ (see \S \ref{sec:cat} below). Thus we get a family of equivalences, one for each $L$, between the same two categories, and this fact plays a key role in our next result.


%
%
%

\subsection{Statement of Proposition \ref{prop:OmExt} and Theorem \ref{thm:FJ}}

The Heisenberg group $H$ is the extension of a symplectic vector space $W$ by a one dimensional center $Z$, and the symplectic group $Sp(W)$ acts on $H$ by automorphisms fixing $Z$. The classical Stone-von Neumann theorem implies that if $\rho$ is an irreducible unitary representation of $H$ with non-trivial central character under $Z$ then $\rho$ extends to a representation $\sigma$ of the universal cover of the symplectic group $Sp(W)$. A celebrated result of Segal-Shale-Weil \cite{Seg,Sha,Wei} asserts that $\sigma$ descends to a double cover of $Sp(W)$ called the metaplectic group $\Mp(W)$. Our next result is a smooth analog of this for the canonical representation $\Omega$.

\begin{introprop}[See \S \ref{sec:Weil}]\label{prop:OmExt} There is a unique extension of $\Omega$ to a representation of $\Mp(W)\ltimes H$.
\end{introprop}

Our proof is pointwise; namely we define the action of $\Mp$ on $\Omega$ by analogy with the formulas for the Weil representation, and check it for each character $\chi$ of $Z$.

We next upgrade these formulas to a functorial equivalence. Let $\Rep^{\infty}(\Mp(W)\ltimes H)_{-}^\times$ denote the subcategory of $\Rep^{\infty}(\Mp(W)\ltimes H)$ consisting of  representations whose restriction to $H$ lies in $\Rep^{\infty}(H)^\times$, and which are \emph{genuine} in the sense that the preimage of the identity of $\Sp$ in $\widetilde{\Sp}$\ acts non-trivially.
We upgrade the functors $C$ and $I$ from Theorem \ref{thmB} to functors
$$C^+:\Rep^{\infty}(\Mp(W)\ltimes H)_{-}^\times\leftrightarrows \Rep^{\infty}(\Sp(W)\times Z)^\times:I^+$$
given by $C^+(\tau):= \tau\hot_H \overline{\Omega}$ with $\Mp(W)$ acting diagonally (thus factorizing through $\Sp(W)$), and $Z$ acting on $\tau$;  and $I^+(\rho):=  \rho\hot_Z \Omega$, with  $H$ acting on $\Omega$, and $\Mp$ acting diagonally.

\begin{introtheorem}[See \S \ref{sec:Weil}]\label{thm:FJ}
The functors $C^+,I^+$ are quasi-inverses; thus $\Rep^{\infty}(\Mp(W)\ltimes H)_{-}^\times$ and $\Rep^{\infty}(\Sp(W)\times Z)^\times$ are equivalent.
\end{introtheorem}
 }

 Since every representation of $Z$ trivially extends to a representation of $\Sp(W)\times Z$ on the same space, we obtain a functor
$$T: \Rep^{\infty}(H)^\times \overset{C}{\to}\Rep^{\infty}( Z)^\times\overset{\mathrm{triv}}{\to} \Rep^{\infty}(\Sp(W)\times Z)^\times\overset{I^+}{\to} \Rep^{\infty}(\Mp(W)\ltimes H)_{-}^\times$$
The functor $T$ can be characterized uniquely as follows.
\begin{introtheorem}[See \S \ref{sec:Weil}]\label{thm:Ext}
The functor $T$ is the unique right quasi-inverse of the restriction functor $$\Rep^{\infty}(\Mp(W)\ltimes H)_{-}^\times \to \Rep^{\infty}(H)^\times.$$
\end{introtheorem}

In other words, Theorem \ref{thm:Ext} says that any $\tau \in \Rep^{\infty}(H)^\times$ extends functorially via $T$ to a representation of $\Mp(W)\ltimes H$ on the same space, and such a functorial  extension is unique. A single representation  $\tau \in \Rep^{\infty}(H)^\times$ can have several extensions, but any two of them will  differ by an action of $\Sp(W)$ on $C(\tau)$.

\subsection{Statement of Theorem \ref{thm:Int_proWhit}}\label{subsec:IntproWHit}

We give an application of our results to degenerate Whittaker models. Let $G$ be a real reductive group with Lie algebra $\fg$. A Whittaker pair $(S,\varphi)$ for $G$ consists of a real semisimple element $S$ in $\fg$ and an element $\varphi \in \fg^*$ such that
$ad^*(S)(\varphi)=-2\varphi$. To this data one can associate the nilpotent Lie algebra  $\fu =\fg^S_{\ge1}$ consisting of $ad(S)$-eigenspaces with eigenvalues $\ge1$, \DimaA{an anti-symmetric form $\omega_{\varphi}$ on $\fu$  
and a Heisenberg quotient $\fh$ of $\fu$ such that $\varphi$ defines a central character of $\fh$ - see \S \ref{sec:Whit} below. 
Let $\fl\subset \fu$ be a maximal isotropic Lie subalgebra, and $L\subset G$ be the corresponding unipotent subgroup. Let $\chi_\varphi$ the unitary character of $L$ defined by $\varphi$. 
  We then define
 $\pro\cW_{S,\varphi}:=\pro\cW_{S,\varphi}^L:=\Sc(G)_{(L,\chi_\varphi)},$ where the subindex $(L,\chi_\varphi)$ means generalized coinvariants - see \S \ref{sec:coinv} below.
We show in \S \ref{sec:Whit} that this representation of $G$ does not depend on the choice of $L$. 
 For any $\pi\in \Rep^{\infty}(G)$ we define
$$\pi_{(S,\varphi)}:=\pi_{(S,\varphi)}^L:=\pi\hot_G (\pro\cW^L_{S,\varphi})$$
We also construct a natural isomorphism $\pi^L_{(S,\varphi)}\cong \pi_{(L,\chi_{\varphi})}$.}
\Dima{If $S$ can be completed to an $\sl_2$-triple $(e,S,f)$ such that $\varphi$ is given by the Killing form pairing with $f$ then we call the pair $(S,\varphi)$ and the model \emph{neutral}.} By the Jacobson-Morozov theorem, every \DimaA{nilpotent $f\in \fg$} can be completed to a neutral pair, uniquely up to conjugation. Thus every \DimaA{nilpotent $\varphi\in \fg^*$} defines a neutral model, which does not depend on the choice of $S$. We will thus denote $\pro\cW^{L}_{S,\varphi}$ for neutral $S$ by $\pro\cW^{L}_{\varphi}$, and the corresponding $\pi^{L}_{(S,\varphi)}$ by $\pi^{L}_{(\varphi)}$. Denote by $\WO(\pi)$ the set of all $\varphi\in \fg^*$ such that $\pi_{(\varphi)}$ does not vanish. This set is a union of coadjoint $G$-orbits, and we denote the collection of orbits in $\WO(\pi)$ that are maximal under the closure ordering by $\WS(\pi)$.

\DimaB{Following \cite{GGS} we construct in \S \ref{sec:Whit} below a map $\pro\cW^{L}_{(\varphi)}\onto \pro\cW^{L'}_{(S,\varphi)}$ for any pair $(S,\varphi)$ and any choice of maximal isotropic Lie algebras $\fl,\fl'$ as above. 
For any $\pi \in \Rep^{\infty}(G)$ this defines a map $\pi^L_{(\varphi)}\onto \pi^{L'}_{(S,\varphi)}$. 
}

In \S \ref{sec:Whit} we use Theorem \ref{thmB}, as well as the technique of \cite{GGS2}, to prove the following theorem.

\begin{introthm}[See \S\ref{sec:Whit}]\label{thm:Int_proWhit}
Let $\pi\in \Rep^{\infty}(G)$, and let $(S,\varphi)$ be a Whittaker pair with $\Dima{G\cdot}\varphi\in \WS(\pi)$. Then the map $\pi^{L}_{(\varphi)}\to \pi^{L'}_{(S,\varphi)}$ given by \eqref{=modSeqpi} below is an \DimaA{isomorphism}.
\end{introthm}

\DimaA{
\subsection{Analogs over other local fields}

Analogous theorems with analogous proofs still hold if one lets $H$ be a Heisenberg group over any local field $F$ of characteristic zero (that is $\R,\C$, or a finite extension of $\Q_p$). In this case $Z$ is the field of definition, and one defines $\Rep^{\infty}(Z)$ to be the category of smooth modules  over the algebra $\Sc(\widehat{Z}^{\times})$ of Schwartz functions on  $\widehat{Z}^{\times}$. 

Over non-Archimedean  fields 
the space of Schwartz functions means just locally constant compactly supported complex-valued functions. This case is easier since in this case no functional analysis is required, and the category of smooth modules is equivalent to the more geometric category of equivariant $l$-sheaves. For certain non-Archimedean analogs of Theorems \ref{thmB}-\ref{thm:Ext} above see \cite[\S 2]{Weis}.
}

\subsection{Structure of the paper}
In \S\S\ref{sec:alg},\ref{sec:Sc} we give the necessary preliminaries from \cite{dCl} on abstract \Fre algebras, and on Schwartz functions on Nash manifolds. In \S \ref{sec:Imp} we recall and slightly extend the du Cloux theory of smooth imprimitivity systems.  In \S\ref{sec:proj1} we develop the technical notion of a projective line over a 1-parameter group.

In  \S\ref{sec:nonsing} we use all the previous sections to define the category $\Rep^{\infty}(Z)^{\times}$ of non-singular representations of $\R$.
 In \S \ref{sec:cat} we define the category  $\Rep^{\infty}(H)^{\times}$ of non-singular representations of $H$ and construct an equivalence between it and $\Rep^{\infty}(Z)^{\times}$, proving Theorem \ref{thmB}.

In \S \ref{sec:Weil} we construct the generalized oscillator representation, and prove Proposition \ref{prop:OmExt}, and Theorems \ref{thm:FJ} and \ref{thm:Ext}.

In \S \ref{sec:coinv} we prove some technical statements on generalized coinvariants.

In \S \ref{sec:Whit} we apply Theorem \ref{thmB}, and the statements from \S \ref{sec:coinv}, to degenerate pro-Whittaker models, and prove Theorem \ref{thm:Int_proWhit}.

In Appendix \ref{app:Kash} we prove Theorem \ref{thmA}.

\subsection{Ideas of the proofs}
Let us first give an intuition for the notion of smooth imprimitivity system, defined in \S \ref{sec:Imp} below. A $G$-imprimitivity system $\cF$ on a $G$-space $X$ is a \Fre space with compatible structures of a module over the algebras of Schwartz functions on $X$ (with pointwise multiplication as product) and on $G$ (with convolution), satisfying certain conditions.

Geometrically, we may think of a $G$-imprimitivity system on a $G$-space $X$ as an equivariant sheaf: for any point $x\in X$ we have a \Fre space  $\cF_x$, and for any $g\in G$ a continuous linear operator $\cF_{x}\to \cF_{gx}$. The space $\cF$ is the space of global sections of this sheaf.

Let us now sketch our proof of Theorem \ref{thmB}. Let $\tau\in \Rep^{\infty}(H)^{\times}$. Using Fourier transform on $L\times Z$, $\tau$ defines an $H$-imprimitivity system  on $\widehat{L}\times \widehat{Z}$, with the property that $L\times Z\subset H$ acts on the fiber at any point $(\chi,\psi)$ by the character $(\chi,\psi)$ .  The superscript $\times$ in $\Rep^{\infty}(H)^{\times}$ means that the jet of this sheaf at \Raul{$\widehat{L}\times 0$}  vanishes, and thus the system is completely determined by its restriction to $\widehat{L}\times \widehat{Z}^{\times}$.
Since $0\times  \widehat{Z}^{\times} \subset \widehat{L}\times \widehat{Z}^{\times}$ is a section transversal to the action of $H$, the stabilizer of any point in this system is $L\times Z$, and the action of this stabilizer is given by the point, the system really contains the same information as a sheaf on $\widehat{Z}^{\times}$, or equivalently a representation $\rho\in \Rep^{\infty}(Z)^{\times}$.

In Proposition \ref{prop:OmExt} we construct the action on the generators of $\Sp(W)$ by analogy with the formulas for the classical Weil representation given in \cite[(4.24)-(4.26)]{Fol}. In these formulas, we replace the character of $Z$ given by $\exp(2\pi i t)$ by the representation $\varpi$.  We prove that this is a representation, and also prove the uniqueness by evaluating each function in $\varpi$ at each point in $\widehat{Z}^{\times}$ and reducing to the classical case.

We prove Theorems \ref{thm:FJ} and \ref{thm:Ext} by arguing that $\Omega\in \Rep^{\infty}(H)^{\times}$ and $\varpi\in \Rep^{\infty}(Z)^{\times}$ generate their corresponding categories, and thus reducing the theorems on functors to theorems on these representations, which in turn are again proved by evaluation at each point in $\widehat{Z}^{\times}$.

\subsection{Acknowledgements}
D.G. was partially supported by ISF grant 249/17, and BSF grant 2019724.
S.S. was partially supported by NSF grants DMS-1939600 and DMS-2001537, and Simons foundation grant 509766.

\section{Preliminaries}
\DimaA{
\subsection{Nuclear \Fre spaces}

Let us begin with a brief recall on some standard facts about nuclear \Fre spaces. For more details we refer the reader to \cite[Appendix A]{CHM}.

If $V$ is a topological vector space, then we denote by $V'$ its topological dual, that is the space
of continuous linear functionals $V\to \bC$. We endow $V'$ with the strong dual topology (i.e. the topology of uniform convergence on bounded sets) and note that
$V'$ is a  separated topological vector space.

\par If $T: V \to W$ is a morphism of topological vector spaces, then we denote by $T': W'\to V'$ the corresponding
dual morphism.  A morphism $T: V\to W$ is called {\it strict}, provided that  $T$ induces an isomorphism
of topological vector spaces $V/\ker T \simeq \Im T$.

In this paper we consider mostly nuclear \Fre spaces, NF-spaces for short. If $F$ is a NF-space and $E\subset F$ is a closed subspace then both $E$ and $F/E$ (in the induced and quotient topologies respectively) are NF-spaces. Moreover, any surjective morphism of \Fre spaces is open (see  \cite[Theorem 17.1]{Tre}). Thus, a morphism of NF-spaces is strict if and only if it has closed image. The following lemma says that morphisms of finite corank are always strict.

\begin{lemma}[{\cite[Lemma A.1]{CHM}}]\label{lem:FinGood}
Let $T:E\to F$ be a continuous linear map between $NF$-spaces. Assume that the image $\Im T$ is \Raul{of} finite codimension in $F$. Then $T$ is a strict morphism.
\end{lemma}

Any closed subset of a \Fre space is a \Fre space, and any quotient of a \Fre space by a closed subspace is a \Fre space. The same holds for NF spaces. If $V$ is a \Fre space and $W$ is a NF space then the projective and the injective topologies on their tensor product coincide. We will denote the completion with respect to these topologies by $V\hot W$.
 The completed tensor product of NF spaces is an NF space, and the functor of tensor product with a fixed NF space is exact in the sense that it sends short exact sequences with strict morphisms to short exact sequences with strict morphisms. 

The category of NF spaces is stable under inverse limits over sequences. 
The duals of NF spaces are called DNF spaces. NF spaces are reflexive, and thus the duality between NF and DNF spaces is an equivalence of categories.}

\subsection{Abstract {\Fre} algebras}\label{sec:alg}

We will start by recalling some definitions and results from \cite[\S 2.3]{dCl}. For completeness, we will include proofs for some of the main results, but further details can be found in the aforementioned paper.

\begin{definition}
 Let $A$ be a Fr\'echet algebra, and let $E$ be a Fr\'echet $A$-module.
  Let $\alp:=\alp_E$ denote the canonical map $\alp_E:A\hot E\longrightarrow E$. We say that   $E$ is
 \begin{enumerate}[(i)]
  \item \emph{Non-degenerate} if
   $\{v\in E\, | \, A\cdot v =0\}=0$
and $\alp$ has dense image.
\item \emph{Differentiable} if it is non-degenerate  and $\alp$ is surjective.
\item \emph{Hereditarily differentiable} if any closed submodule $F\subset E$ is differentiable.
\item \emph{Factorizable} if it is non-degenerate  and  the restriction of $\alp$ to the algebraic tensor product $A\otimes E$ is surjective.
 \end{enumerate}

\end{definition}

\begin{definition}
 We say that a Fr\'echet algebra $A$
\DimaA{ is \emph{differentiable} if  it is differentiable as a module over itself. We say that $A$ 
  has the \emph{factorization property} if it  is differentiable  and every differentiable $A$-module is factorizable.} We say that $A$ is \emph{hereditary} if any differentiable module is hereditarily differentiable.
\end{definition}

\begin{remark}
 Notice that if $A$ is hereditary and has the factorization property then it is \emph{hereditarily factorizable,} that is every closed submodule $F$ of a factorizable module $E$ is factorizable.
\end{remark}
\Dima{
\begin{defn}
Let $A$ be a Fr\'echet algebra and $E$ be a non-degenerate \Fre $A$-module.
The \emph{space of Schwartz vectors} $E(A)$ of $E$\ is the image of $\alp_E$, with the topology that makes the  linear isomorphism $(A\hot E)/\Ker\alp_E\cong E(A)$ a homeomorphism.
\end{defn}

\begin{lemma}
 Let $A$ be a \DimaA{differentiable} Fr\'echet algebra. Then
 \begin{enumerate}[(i)]
  \item If $F$ is a \Fre space, then $A\hot F$ is differentiable.
  \item If $E$ is a non-degenerate Fr\'echet module, then the space of Schwartz vectors $E(A)$ is differentiable.
 \end{enumerate}
\end{lemma}
}
\begin{proof}
 For the first part of the lemma, we notice that, since $A$ is differentiable, the map $\alp_A:A\hot A\rightarrow A$ is surjective. Hence, the map $\alp_{A\hot E}: A\htimes(A\htimes F)\rightarrow A\htimes F$ is also surjective.

 For the second part, we notice that $E(A)$ is a quotient of $A\htimes E$ and hence the result follows from the above argument.
\end{proof}

\begin{definition}\label{def:AI}
 An \emph{approximate identity} on a \Fre algebra $A$ is a sequence $\{e_{n}\}$ in $A$ such that for all $a\in A$ we have \DimaA{that
 \[
  \lim_{n\rightarrow \infty} e_{n}\cdot a = \lim_{n\rightarrow \infty} a\cdot e_{n}=a.
 \]}
\end{definition}

\begin{lemma}\label{lem:App1}
 If $A$ is a \Fre algebra with an approximate identity $\{e_{n}\}$, and $E$ is a differentiable $A$-module, then
 \[
  \lim_{n\rightarrow \infty} e_{n}\cdot v = v \qquad \mbox{for all $v\in E.$}
 \]
\end{lemma}

\begin{proof}
 Since every differentiable module is a quotient of a module of the form $A\htimes F$, it is enough to consider just spaces of  this form.
 From \cite[Theorem 45.1]{Tre}, any $w\in A\htimes F$ can be written in the form
 \[
  w=\sum_{j=1}^{\infty} \lambda_{j}a_{j}\otimes v_{j},
 \]
where $\sum_{j=1}^{\infty} |\lambda_{j}|\leq 1$ and $a_{j}$, $v_{j}$ are sequences that converge to $0$ in $A$ and $F$, respectively. In particular, the set $\{a_{j}\}\cup \{0\}$ is compact, and hence $e_{n}a_{j}\rightarrow a_{j}$ uniformly. It follows that
\[
 e_{n}w=\sum_{j=1}^{\infty} \lambda_{j} e_{n}a_{j}\otimes v_{j}\rightarrow \sum_{j=1}^{\infty} \lambda_{j}a_{j}\otimes v_{j}=w.
\]
\end{proof}

\begin{lemma}\label{lemma:A_approx_identity}
 Let $A$ be a differentiable Fr\'echet algebra with an approximate identity, and $E$ be a non-degenerate $A$-module. If we set
 \[
  \mathcal{F}(E)=\{F\subset E\,| \, \mbox{$F$ is closed} \}
 \]
 and
 \[
  \mathcal{F}(E(A))=\{F\subset E(A)\,| \, \mbox{$F$ is closed}\},
 \]
Then the map $F\mapsto \overline{F}$ from $\mathcal{F}(E(A))$ to $\mathcal{F}(E)$\Dima{, that sends each submodule into its closure in $\cF(E)$,} is injective, and increasing with respect to the natural order on $\mathcal{F}(E(A))$ and $\mathcal{F}(E)$.
\end{lemma}

\begin{proof}
 It's clear that the map $F\mapsto \overline{F}$ is increasing. Now let $F$ be a closed submodule of $E(A)$. Then the \Dima{restriction of $\alp_E$\ to $A\hot F$ is a continuous map from $A\htimes F$} onto $F$. Now, since $A\htimes F$ is dense in $A\htimes \overline{F}$ it follows that
\[
 \Dima{\alp_E}(A\htimes \overline{F})=\overline{F}(A)\subset F.
\]
 Now, since $E$ is non-degenerate, we have that if $A\cdot v=0$ for $v\in \overline{F},$ then $v=0$. On the other hand, since $E(A)$ is differentiable, we have, from Lemma \ref{lem:App1}, that $e_{n}\cdot v \rightarrow v$ for all $v\in E(A)$ and hence for all $v\in F.$ Consequently, $F$ and $\overline{F}$ are two non-degenerate $A$-modules, and hence $\overline{F}(A)$ is differentiable. But then
 \[
  \overline{F}(A)=\Dima{\alp_E}(A\htimes \overline{F}(A))\subset \Dima{\alp_E}(A\htimes F)=F(A).
 \]
It follows that $F(A)=\overline{F}(A).$ Therefore, we can recover $F$ from $\overline{F}$ as the closure in $E(A)$ of $\overline{F}(A)$. From this we conclude that the map $F\mapsto \overline{F}$ has to be injective.
\end{proof}

\Dima{
\begin{defn}
We say that a Fr\'echet $A$-module $E$ is \emph{topologically irreducible} if every non-zero submodule is dense in $E$.
\end{defn}}

\begin{corollary}
 If $E$ is topologically irreducible as an $A$-module, then so is $E(A).$
\end{corollary}

\begin{lemma}\label{lemma:B_A-algebra_with_approximate_identity} 
 Let $B$ be a \Fre $A$ algebra with an approximate identity. 
\DimaA{ Suppose that $B$ is differentiable as a module over itself.
 Then any closed $B$-submodule of $B\htimes E$ is $A$-invariant. In particular, if $E$ is a differentiable $B$-module, then the natural action of $A$ on $B\htimes E$ descends to an action on $E$.}
\end{lemma}

\begin{proof}
 First notice that if $w=\tilde{b}\otimes v \in B\otimes E$, then for all $a\in A$, $b\in B$ we have that
 \[
  a(bw)=a(b\tilde{b})\otimes v=(ab)\tilde{b}\otimes v = (ab)w.
 \]
Since linear combinations of pure tensors are dense in $B\htimes E$, we conclude that for all $a\in A,$ $b\in B$ and $w\in B\htimes E$ we have that
\[
 a(bw)=(ab)w.
\]
Now let $F\subset B\htimes E$ be a closed $B$-submodule and let $(e_{n})_{n}$ be an approximate identity in $B$. Then, for all $w\in F$ and $a\in A$ we have, according to Lemma \ref{lem:App1}, that
\begin{align*}
 a\cdot w  = a(\lim_{n\rightarrow \infty} e_{n}\cdot w)
           = \lim_{n\rightarrow \infty} a(e_{n}\cdot w)
           = \lim_{n\rightarrow \infty} ((a e_{n})\cdot w) \in F
\end{align*}
since $ae_{n}\in B$ for all $n$ and $F$ is a closed $B$-submodule. Letting $F$ be the kernel of the multiplication map $B\htimes E\rightarrow E$ we see that if $E$ is differentiable, then the action of $A$ on $B\htimes E$ indeed descends to an action of $A$ on $E$.
\end{proof}

\subsection{Schwartz functions on Nash manifolds}\label{sec:Sc}
We will use the theory of Schwartz functions and tempered distributions on Nash  manifolds, see e.g., \cite{dCl,AG_Sc}.
Nash manifolds are smooth semi-algebraic manifolds, e.g. a union of connected components of the manifold of real points of an algebraic manifold defined over $\R$.

We denote the space of complex valued Schwartz functions on a Nash manifold $Y$ by $\Sc(Y)$. This space has a natural structure of a commutative \Fre algebra\DimaA{, given by pointwise product of functions}.
\DimaA{If $Y$ is affine, then $\Sc(Y)$ consists of infinitely smooth functions $f$ on $Y$ such that $|Df|$ is bounded for every algebraic differential operator $D$ on $Y$. The \Fre topology is given by the family of seminorms $\sigma_D(f)=\sup_{y\in Y} |Df(y)|.$}
 For a Nash group $G$, with fixed left-invariant measure $dg,$ $\Sc(G)$ also has a natural structure of an algebra under convolution, which is not commutative unless $G$ is commutative.
\DimaA{If $X$ is compact then $\Sc(X)$ is $C^{\infty}(X)$. For any $X$, $\Sc(X)$ is a nuclear \Fre space.}

\begin{thm}[{\cite[Theorem 1.2.4(i)]{dCl} or \cite[Theorem 5.4.3]{AG_Sc}}]\label{thm:U}
Let $X$ be a Nash manifold and $U\subset X$ be an open Nash subset. Then the extension by zero defines an isomorphism of $\Sc(U)$ with the \DimaA{closed} subspace of $\Sc(X)$ consisting of functions that vanish to infinite order outside $U$.
\end{thm}

\DimaA{

\begin{thm}[{\cite[Corollary 2.6.3]{AGRhamShap}}]\label{thm:prod}
Let $X,Y$ be Nash manifolds. Then $\Sc(X)\hot \Sc(Y)\cong \Sc(X\times Y)$.
\end{thm}
}

\begin{theorem}[{\cite[\S 2.3]{dCl}}]
 If $G$ is a Nash group and $X$ is a Nash manifold, then both $\Sc(G)$ and $\Sc(X)$ are hereditary, have the factorization property and an approximate identity.
\end{theorem}

\begin{lemma}[{\cite[Example 2.3.3]{dCl}}]
Let $G$ be a Nash group. Given a smooth, moderate growth, representation of $G$ on a Fr\'echet space $E,$ we can define a natural action of $\Sc(G)$ on this space that makes it a differentiable module. Furthermore, \DimaA{this defines an equivalence of categories between smooth, moderate growth, representations of $G$ in \Fre spaces and differentiable Fr\'echet  $\Sc(G)$-modules.}
\end{lemma}
\DimaA{We denote these equivalent categories by $\Sc\mod(G)$, and the subcategory consisting of nuclear \Fre modules by $\Rep^{\infty}(G)$.} If $Z$ is an arbitrary Nash manifold, we will denote by $\Sc\mod(Z)$ the category of differentiable Fr\'echet $S(Z)$-modules.

\subsection{Imprimitivity systems}\label{sec:Imp}

In this section we recall and slightly extend du-Cloux theory of smooth imprimitivity systems from \cite[\S 2.4-2.5]{dCl}.

\begin{definition}
Let $X$ be a Nash $G$-manifold. An \emph{imprimitivity system of base $X$} on a differentiable $G$-module $E$ is an action of the algebra $\Sc(X)$ such that
\[
(g\cdot \alp) v= g (\alp (g^{-1} v)), \qquad \mbox{for all $g\in G,$ $\alp\in \Sc(X),$ $v\in E.$}
\]
We will denote the category of imprimitivity systems for $(G,X)$ by $\Rep_{G,X}$ and consider the categories $\Rep_{G,X}^{nd}$ and $\Sc\mod_{G,X}$ consisting of the systems for which the action of $\Sc(X)$ on $E$ is non-degenerate and differentiable, respectively.
\end{definition}
Intuitively, we think of an imprimitivity system on $X$ as of a $G$-equivariant cosheaf on $X$.
\\
Denote $\cA:=\Sc(G\times X)$.  The space $\cA$ has a natural structure of an imprimitivity system given by
\begin{equation}
(g\cdot \alp)(g_1,x):=\alp(g^{-1}g_1,x)\Dima{\text{ and } (\psi\cdot \alp)(g,x)=\psi(gx)\alp(g,x)},
\end{equation}
 for all $g,g_1\in G, x\in X, \alp\in \cA,\psi\in \Sc( X)$.
Given an imprimitivity system $E$, we define a natural map $\cA\otimes E\to E$ by
\begin{equation}\label{=Aim}
\alp\otimes v \mapsto \int_G(\alp^g \cdot gv)dg, \text{ where }\alp^g\in \Sc(X) \text{ is given by } \alp^g(x):=\alp(g,g^{-1}x)
\end{equation}
It is easy to check that substituting $\cA$ for $E$, we obtain an algebra structure on $\cA$. Explicitly, the
  multiplication is given by:
\[
\alp_{1}\ast\alp_{2}(g,x) =\int_{G}\alp_{1}(g_{1},g_{1}^{-1}gx)\alp_{2}(g_{1}^{-1}g,x)\,dg_{1}.
\]
Furthermore, for any imprimitivity system $E$, \eqref{=Aim} defines an $\cA$-module structure on $E$.

\begin{lemma}
\begin{enumerate}[(i)]
\item The \Fre algebra $\cA$ is differentiable, hereditary, has an approximate identity and satisfies the factorization property.
\item The $\cA$-module structure \eqref{=Aim} defines an equivalence between the category $\Sc \mod_{G,X}$ of smooth imprimitivity systems and the category of differentiable $\cA$-modules.
\end{enumerate}
\end{lemma}
\begin{proof}
We will start by showing $\cA\in \Sc\mod_{G,X}$. Since $\cA=\Sc(G\times X)\cong \Sc(G) \hot \Sc(X)$ as $G$-modules, $\cA\in \Sc\mod(G)$. After the change of variables $(g,x)\mapsto (g,g^{-1}x)$, the action of $\Sc(X)$ on $\cA$ becomes multiplication on the $\Sc(X)$ coordinate. Thus $\cA\in \Sc\mod(\Sc(X))$ and thus $A\in \Sc\mod_{G,X}$.

Let us now show that any $E\in \Sc\mod_{G,X}$ is differentiable as an $\cA$-module. Let $\{x_n\}$ and $\{\theta_n\}$ be approximate identities in $\Sc(G)$ and $\Sc(X)$ respectively, and let $e_n(g,x):=x_n(g)\theta_n(x)$. Then, by Lemma \ref{lem:App1},
\DimaA{
\[
 \lim_{n\to \infty} \theta_n\cdot v = v,\quad \lim_{n\to \infty} x_n\cdot v = v,
\]
and thus, by the standard bilinear argument, we conclude that
\begin{equation}\label{=1v}
 \lim_{n\to \infty} e_n\cdot v = v\quad \text{ for any } v\in E.
\end{equation}}
Thus, $E$ is non-degenerate. Taking $\cA$ as $E$ we also obtain that $e_n$ is an approximate identity in $\cA$. Finally, since $\Sc(G)$ and $\Sc(X)$ satisfy the factorization property, we have $E  =\Sc(X)E=\Sc(G)\Sc(X)E$ and hence $AE=E$, that is $E$ is differentiable as an $\cA$-module.

Now let $E\in \Sc\mod(\cA)$, and let $V$ denote the kernel of the  epimorphism $m:\cA\hot E\onto E$. The imprimitivity system on $\cA$ defines an imprimitivity system on $\cA\hot E$ by action on the $\cA$ coordinate. Then $V$ is a subsystem, {\it i.e.} is stable under multiplication by $\Sc(X)$ and $\Sc(G)$. Indeed, for any $v\in V, x\in \Sc(X)$, and $\theta\in \Sc(G)$ we have
$$x\cdot v=\lim_{n\to \infty} xe_nv\in V \quad \text{ and }\theta\cdot v=\lim_{n\to \infty} \theta e_nv\in V.$$
Now, we define the imprimitivity system on $E$ using the identification $m:(\cA\hot E)/V \cong E$.

Finally, since $\Sc(X)$ and $\Sc(G)$ are hereditary and satisfy the factorization property, so does $\cA$.
\end{proof}

\begin{lemma}
If $E\in \Rep_{G,X}$ then $E(\cA)=E(\Sc(X))$.
\end{lemma}
\begin{proof}
Since $\cA$ is a differentiable $\Sc(X)$-module, it follows that $\cA\hot E$ is differentiable for $\Sc(X)$, and hence the same is true for $E(\cA)$. Thus $E(\cA)\subset E(\Sc(X))$.

To prove the opposite inclusion, notice that $\Sc(X)\hot E\to E$ is a $G$-morphism, if we act on $\Sc(X)\otimes E$ diagonally. It follows that $E(\Sc(X))$ is a $G$-module, and hence $E(\Sc(X))\in \Sc\mod_{G,X}$. Then by the definition of $E(\cA)$ it follows that $ E(\Sc(X))\subset E(\cA)$.
\end{proof}

\subsection{Imprimitivity theorem}\label{subsec:ImpThm}

Let $H\subset G$ be a Nash \Raul{closed} subgroup, and  let $X=H\backslash G$. We will also
let $x_0\in X$ denote the coset $H$, and let $m_{x_0}:=\{\alp\in \Sc(X)\, \vert \alp(x_0)=0\}$. For any $E\in \Sc\mod(G,X)$ define $E_0:=E/m_{x_0}E$. The subspace $m_{x_0}E$ is closed in $E$ by \cite[Lemma 2.5.7]{dCl}.
\footnote{du-Cloux defines the ideal $m_{x_0}$ to be in $\Sc(X)\oplus \C$, since he talks about $\Rep_{G,X}$. For $E\in \Sc\mod(G,X)$, we can use $m_{x_0}\subset \Sc(X)$, since $E=\Sc(X)E$.}

\begin{defn}[{ \cite[\S 2]{dCl}}]\label{def:ind}
 Let  $\pi \in  \Sc\mod(H)$, and let $\Sc(G,\pi)$ denote the space of Schwartz functions from $G$ to the underlying space of $\pi.$ We define a map from  $\Sc(G,\pi)$  to the space $C^{\infty}(G,\pi)$ of all smooth $\pi$-valued functions on $G$ by $f\mapsto \overline f$, where $$\overline f (x)=\int_{h\in H}\pi(h)f(xh)dh,$$
and $dh$ denotes a fixed left-invariant measure on $H$. The Schwartz induction $\ind_H^G(\pi)$ is defined to be the image of this map. Note that $\ind_H^G(\pi)\in \Sc\mod(G)$.
\end{defn}

Given an $H$-module $F$, we denote by $F^{-1/2}$ the tensor product of $F$\ with the square root of the inverse modular function on $H$. Similarly, we define $\delta^{1/2}\ind_H^GF$ to be the Schwartz induction from $H$ to $G$, tensored with the square root of the modular function of $G$.

\begin{theorem}[{\cite[Theorem 2.5.8]{dCl}}]\label{thm:imprimitivity0}
The functor $E\mapsto E_0^{-1/2}$ defines an equivalence of categories between $\Sc\mod(G,X)$ and $\Sc\mod(H)$ with inverse $F\mapsto \delta^{1/2}\ind_H^GF$.
In particular, the natural $G$-morphism $E\to \delta^{1/2}\ind_H^GE_0$ is $\Sc(X)$-equivariant.
\end{theorem}

\begin{theorem}\label{thm:imprimitivity}
Let $Z$ be a Nash manifold, and $M:=X\times Z$, with $X$ given as above. Let $G$ act on $M$ by acting on the first coordinate.

\begin{enumerate}[(i)]

\item \label{imp:Smod} If $E\in \Sc\mod(G,M)$ then $E\in \Sc\mod(G,X)$ and $E\in \Sc\mod(Z)$.
\item \label{imp:variant} The natural map $E\to E_0^{-1/2}$ is $\Sc(Z)$-equivariant.

\item \label{imp:ence} The functor $E\to E_0^{-1/2}$ defines an equivalence of categories between $\Sc\mod(G,M)$ and $\Sc\mod(H,Z)$, with inverse $F\mapsto \delta^{1/2}\ind_H^GF$.
\end{enumerate}
\end{theorem}

\begin{proof}
\eqref{imp:Smod} We need to show that $E$ is differentiable for $\Sc(X)$ and $\Sc(Z)$. Since $E$ is differentiable for $\Sc(M)$, the natural map $\Sc(M)\hot E\to E$ is surjective. Since $\Sc(M)\cong \Sc(X)\hot \Sc(Z)$, $\Sc(M)\hot E$ is a differentiable module for $\Sc(X)$ and $\Sc(Z)$. Thus, it is enough to show that any closed $\Sc(M)$ submodule $V\subset \Sc(M)\hot E$ is closed under multiplication by $\Sc(X)$ and $\Sc(Z)$. This is shown using an approximate identity $\{e_n\}$ in $\Sc(M)$.

\eqref{imp:variant} is evident, since the action of $\Sc(Z)$ commutes with that of $\Sc(X)$.

\eqref{imp:ence} Follows from (\ref{imp:Smod},\ref{imp:variant}) and Theorem \ref{thm:imprimitivity0}.
\end{proof}

\Dima{
\begin{defn}\label{def:coinv}
For $\pi\in \Rep^{\infty}(G)$, denote by $\pi_G$ the space of coinvariants, i.e. quotient of $\pi$ by the intersection of kernels of all $G$-invariant functionals. Explicitly,
$$\pi_G=\pi/\overline{\{\pi(g)v -v\, \vert \,v\in \pi, \, g\in G\}}.$$
For $\tau \in \Rep^{\infty}(G)$ we denote  $\pi\hot_G\tau:=(\pi\hot\tau)_{\Delta G}$ - the space of coinvariants under the diagonal action.
\end{defn}
Note that if $G$ is connected then $\pi_G=\pi/\overline{\fg\pi}$ which in turn is equal to the quotient of $\oH_0(\fg,\pi)$ by the closure of zero, where $\fg$ denotes the Lie algebra of $G$.
}

We will need the following version of Frobenius reciprocity for Schwartz induction.
\begin{lem}[{\cite[Lemma 2.8]{GGS2}}]\label{lem:Frob} Let $\tau \in \Rep^{\infty}(H)$, $\pi\in \Rep^{\infty}(G)$.
 Then
$$\pi\underset{H}{\hot} \tau\delta_H\delta_G^{-1}\cong\pi \underset{G}{\hot} \ind_H^G(\tau).$$
\end{lem}

\section{The projective line over a 1-parameter group}\label{sec:proj1}

Let $Z$ be a 1-parameter group with Lie algebra $\z=\Lie(Z)$, and let $\hat{Z}$ be its associated group of (unitary) characters with Lie algebra $\hat{\z}=\Lie(\hat{Z}).$ We notice that we can naturally identify
\begin{align*}
  \hat{\z} & =\Hom_{\R}(\z,i\R)
            = \{\chi:\z \rightarrow i\R\,|\,\mbox{$\chi(\lambda z)=\lambda\chi(z)$ for $\lambda \in \R$}\} \\
           & = \{\chi:\z^{\times}\rightarrow i\R\,|\, \mbox{$\chi(\lambda z)=\lambda\chi(z)$ for $\lambda \in \R^{\times}$}\}
\end{align*}
where $\z^{\times}=\z-{0}$ and $\R^{\times}=\R-{0}$. We now define
\[
 \hat{\z}_{\sigma}=\{\psi:\z^{\times}\rightarrow i\R \, | \, \psi(\lambda z)=\lambda^{-1}\psi(z)\}
\]
and notice that $\hat{\z}_{\sigma}$ has a natural structure as a real vector space. We let $\hat{Z}_{\sigma}$ be its associated $1$-parameter Lie group, and define
\[
 Z_{\sigma}=[\hat{Z}_{\sigma}]^{\hat{}}
\]
Then, by Pontryagin duality, we have a natural isomorphism $[Z_{\sigma}]^{\hat{}}\simeq \hat{Z}_{\sigma}$. We will set
\[
 \hat{Z}^{\times} = \{\chi \in \hat{Z} \, | \, \mbox{$ d\chi(z)\neq 0$ for $z\in \z^{\times}$}\}
\]
and
\[
 \hat{Z}_{\sigma}^{\times} = \{\psi \in \hat{Z}_{\sigma} \, | \, \mbox{$ d\psi(z)\neq 0$ for $z\in \z^{\times}$}\}.
\]
We remark that, by the definition of $\hat{\z}_{\sigma}$,
\[
 d\psi(\lambda z)=\lambda^{-1} d\psi(z) \qquad \mbox{for all $\psi\in \hat{Z}_{\sigma},$ $z \in \z^{\times},$ $\lambda \in \R^{\times}.$}
\]
Hence we can define, for any $\chi \in \hat{Z}^{\times}$, a character $\sigma\chi \in \hat {Z}^{\times}_{\sigma}$ such that
\[
 d(\sigma\chi)(z)=1/d\chi(z) \qquad \mbox{ for all $z\in \z^{\times}$,}
\]
and it's straightforward to check that this map defines an isomorphism of Nash manifolds $\sigma: \hat{Z}^{\times} \rightarrow \hat{Z}^{\times}_{\sigma}.$ Furthermore, we can use this map to define an isomorphism of Fr\'echet algebras $\sigma: \Sc(\hat{Z}^{\times})\rightarrow \Sc(\hat{Z}_{\sigma}^{\times})$ given by
\[
 (\sigma \phi)(\psi):=\phi(\sigma^{-1}(\psi)) \qquad \mbox{for $\phi\in \Sc(\hat{Z}^{\times})$, $\psi \in \hat{Z}^{\times}_{\sigma}$.}
\]
Using these isomorphisms, we can define the \emph{$\hat{Z}$-projective line} to be the Nash manifold
\[
 P^{1}(\hat{Z})=\hat{Z}+_{\hat{Z}^{\times}}\hat{Z}_{\sigma}
\]
with associated Fr\'echet algebra
\[
 \Sc(P^{1}(\hat{Z}))=\Sc(\hat{Z})+_{\Sc(\hat{Z}^{\times})}\Sc(\hat{Z}_{\sigma}).
\]
Now notice that on the space $\Sc(\hat{Z}^{\times}) \subset \Sc(\hat{Z})$ , $\Sc(\hat{Z}_{\sigma})$ we have a natural action of the groups $Z$ and $Z_{\sigma}$. In the rest of this section we will describe the relation between the infinitesimal actions of these two groups. We start by noticing that for $z \in \z,$ $\phi\in \Sc(\hat{Z}^{\times}),$
\[
 (z\cdot \phi)(\chi)=d\chi(z)\phi(\chi),
\]
and hence $z$ defines an invertible operator on $\Sc(\hat{Z}^{\times})$. We now want to define an isomorphism of Nash manifolds $\sigma:\z^{\times}\rightarrow \z^{\times}_{\sigma}$ in the following way: First we identify
\begin{align*}
 \z_{\sigma} & = \{\alpha:\hat{\z}_{\sigma}\rightarrow i\R \, | \, \mbox{$\alpha(\lambda\psi)=\lambda\alpha(\psi)$ for $\psi \in \hat{\z}_{\sigma}$, $\lambda \in \R.$}\} \\
  & \simeq \{\alpha:\hat{\z}_{\sigma}^{\times} \rightarrow i\R \, | \, \mbox{$\alpha(\lambda\psi)=\lambda\alpha(\psi)$ for $\psi \in \hat{\z}_{\sigma}^{\times}$, $\lambda \in \R^{\times}.$}\}
\end{align*}
and then we define, for $\Raul{z}\in \z^{\times},$ \Raul{$\psi\in \hat{\z}_{\sigma}^{\times}$}
\[
 [\sigma(z)](\psi)=\psi(z).
\]
It's then clear that, for all $\lambda \neq 0,$
\[
 [\sigma(z)](\lambda \psi)=\lambda\psi(z)=\lambda[\sigma(z)](\psi)
\]
and hence $\sigma(z)$ can be identified with an element of $\z_{\sigma}^{\times}.$ Furthermore, for $\lambda \neq 0$
\begin{align*}
 [\sigma(\lambda z)](\psi) & = \psi(\lambda z)
                            = \lambda^{-1}\psi(z)
                            = \lambda^{-1}[\sigma(z)](\psi),
\end{align*}
that is, $\sigma:\z^{\times}\rightarrow \z^{\times}_{\sigma}$ is indeed an isomorphism satisfying $\sigma(\lambda z)=\lambda^{-1} \sigma(z).$  With these observations in place, we now notice that if $\phi\in \Sc(\hat{Z}^{\times})$ and $z\in \z^{\times},$ then for all $\psi \in \hat{Z}^{\times}_{\sigma}$
\begin{multline*}
 [\sigma(z)\sigma(\phi)](\psi)  =d\psi(z) \sigma(\phi)(\psi)  =\frac{1}{\frac{1}{d\psi(z)}} \phi(\sigma^{-1}(\psi))  = \frac{1}{d\sigma^{-1}(\psi)(z)} \phi(\sigma^{-1}(\psi))
         =\\ (z^{-1} \cdot \phi)(\sigma^{-1}(\psi))
         = [\Raul{\sigma(z^{-1} \cdot \phi)](\psi)},
\end{multline*}
that is,
\begin{equation}\label{eq:sigma_z_action}
 \sigma(z)\cdot \sigma(\phi) = \sigma(z^{-1} \cdot \phi),
\end{equation}
where $z^{-1}$ represents the inverse of the operator $z$ acting on $\Sc(\hat{Z}^{\times})$.

\begin{definition}
 Let $X$ be a Nash manifold and let $\Sc(X)$ be its associated Schwartz space. Given an $\Sc(X)$-module $E$ and an open semialgebraic set $U\subset X,$ we define $E_{U}\subset E$ to be the image of the multiplication map $\Sc(U)\htimes E\rightarrow E$ with the topology taken as a quotient of $\Sc(U)\htimes E$.
\end{definition}

\begin{lemma}\label{lem:P1}
 If $E$ is a differentiable $\Sc(P^{1}(\hat{Z}))$-module, then
 \begin{equation}\label{eq:P1_module_decomposition}
  E\simeq E_{\hat{Z}} +_{E_{\hat{Z}^{\times}}} E_{\hat{Z}_{\sigma}}
 \end{equation}

\end{lemma}

\begin{proof}
 Notice that the natural inclusions $E_{\hat{Z}} \hookrightarrow E$ and $E_{\hat{Z}_{\sigma}}  \hookrightarrow E$ define a map
 \[
  \Psi: E_{\hat{Z}} \oplus E_{\hat{Z}_{\sigma}} \rightarrow E.
 \]
 Let us first show that this map is surjective. Let $\{\alp_{1},\alp_{2}\}$ be a partition of unity subordinated to the cover $\{ \hat{Z}, \hat{Z}_{\sigma} \}$ of $P^{1}(\hat{Z})$, that is,
 \begin{enumerate}
  \item $\alp_{1},\alp_{2}\in \Sc(P^{1}(\hat{Z}))$
  \item $\supp \alp_{1} \in \hat{Z}$ and $\supp \alp_{2} \in \hat{Z}_{\sigma}$
  \item $\alp_{1}+ \alp_{2}=1.$
 \end{enumerate}
Then we can define a map
\[
 \Phi:E \rightarrow E_{\hat{Z}} \oplus E_{\hat{Z}_{\sigma}}
\]
given by $\Phi(v)=(\alp_{1}v, \alp_{2}v)$. Then it's clear that for all $v\in V$,
\begin{align*}
 \Psi(\Phi(v)) & = \Psi(\alp_{1}v,\alp_{2}v)  = \alp_{1}v + \alp_{2} v  = 1\cdot v=v
\end{align*}
where the last equality follows from the fact that $E$ is differentiable. From these equations we conclude that $\Psi$ is indeed surjective. On the other hand, the Kernel of $\Psi$ is naturally isomorphic to $E_{\hat{Z}} \cap E_{\hat{Z}_{\sigma}}$ so, in order to prove (\ref{eq:P1_module_decomposition}), we need to show first that $E_{\hat{Z}} \cap E_{\hat{Z}_{\sigma}}=E_{\hat{Z}^{\times}}$. Notice first that, since $\Sc(\hat{Z}^{\times}) \subset \Sc(\hat{Z}),  \Sc(\hat{Z}_{\sigma})$ we have that
\[
 E_{\hat{Z}^{\times}}\subset E_{\hat{Z}}, E_{\hat{Z}_{\sigma}},
\]
and hence
\begin{equation}\label{eq:contained_in_intersection}
 E_{\hat{Z}^{\times}}\subset E_{\hat{Z}}\cap E_{\hat{Z}_{\sigma}}.
\end{equation}
Now let $v\in E_{\hat{Z}}\cap E_{\hat{Z}_{\sigma}}$. Since $ \Sc(\hat{Z})$ satisfies the factorization property, it follows that there exist $\phi_{1},\ldots,\phi_{m}\in  \Sc(\hat{Z})$ and $v_{1},\ldots v_{m}\in E$ such that
\[
 v=\phi_{1}v_{1}+\cdots+\phi_{m}v_{m}.
\]
Similarly, there exists $g_{1},\ldots,g_{l}\in  \Sc(\hat{Z}_\sigma)$ and $w_{1},\ldots w_{l}\in E$ such that
\[
 v=g_{1}w_{1}+\cdots+g_{l}w_{l}.
\]
Hence,
\[
 v=\alp_{1}v+\alp_{2}v=\alp_{1}g_{1}w_{1}+\cdots+\alp_{1}g_{l}w_{l} + \alp_{2}\phi_{1}v_{1}+\cdots+\alp_{2}\phi_{m}v_{m}.
\]
But $\alp_{1}g_{k}\in  \Sc(\hat{Z})\cdot \Sc(\hat{Z}_{\sigma})=\Sc(\hat{Z}^{\times})$ for $k=1,\ldots,l$ and, similarly, $\alp_{2}\phi_{j}\in   \Sc(\hat{Z}_{\sigma}) \cdot \Sc(\hat{Z})=\Sc(\hat{Z}^{\times})$ for $j=1,\ldots,m.$ It follows that $v\in  E_{\hat{Z}^{\times}}$ and, since $v\in E_{\hat{Z}}\cap E_{\hat{Z}_{\sigma}}$ was arbitrary, it follows that
\begin{equation}\label{eq:contained_in_Z_times}
 E_{\hat{Z}}\cap E_{\hat{Z}_{\sigma}} \subset E_{\hat{Z}^{\times}}.
\end{equation}
From equations (\ref{eq:contained_in_intersection}) and (\ref{eq:contained_in_Z_times}) it follows that $E_{\hat{Z}}\cap E_{\hat{Z}_{\sigma}}=E_{\hat{Z}^{\times}}$ as we wanted to show. From this we notice that the maps $\Psi$ and $\Phi$ induce maps
\[
 \tilde{\Psi}:E_{\hat{Z}} +_{E_{\hat{Z}^{\times}}} E_{\hat{Z}_{\sigma}}\longrightarrow E
\]
and
\[
 \tilde{\Phi}:E \longrightarrow  E_{\hat{Z}} +_{E_{\hat{Z}^{\times}}} E_{\hat{Z}_{\sigma}}
\]
that are each other inverses, that is, $E\simeq E_{\hat{Z}} +_{E_{\hat{Z}^{\times}}} E_{\hat{Z}_{\sigma}}$ as we wanted to show.
\end{proof}

\section{Non-singular representations of $\R$}\label{sec:nonsing}

\begin{definition}
 Given a nuclear, Fr\'echet space $V,$ we say that an operator $A\in \End(V)$ is \emph{integrable} if there exists a smooth representation $(\rho,V)$ of a 1-parameter group $Z\cong \R$ and an element $z\in \z=\Lie(Z)$ such that
 \[
  d\rho(z)=A.
 \]
\end{definition}
Smooth here means that $(\rho,V)\in \Sc\mod(Z)$.

\begin{definition}
 Given a 1-parameter group $Z$ with Lie algebra $\z=\Lie Z,$  we say that a smooth representation $(\rho, V)\in \Sc\mod(Z)$ is \emph{non-singular} if $d\rho(z)\in \End(V)$ is invertible with integrable inverse.
\end{definition}

\begin{lemma}\label{lem:nonsin}
 Given a smooth representation $(\rho,V)$ of a 1-parameter group $Z$, the following conditions are equivalent:
 \begin{enumerate}
  \item $(\rho, V)$ is non-singular.
  \item $(\rho, V)$ extends to a representation $\overline{\rho}$ of $\Sc(P^{1}(\hat{Z}))$ and $V$ is differentiable as an $\Sc(\hat{Z}_{\sigma})$-module.
  \item $V$ is differentiable as an $\Sc(\hat{Z}^{\times})$-module.
 \end{enumerate}

\end{lemma}

\begin{proof}
 We start by noticing that $(\rho,V)$ being non-singular is equivalent to the following condition:

 \noindent(1') There exists a smooth representation $(\rho_{\sigma},V)$ of $Z_{\sigma}$ such that
 \[
  d\rho_{\sigma}(\sigma(z))=[d\rho(z)]^{-1} \qquad \mbox{for all $z\in \z^{\times}.$}
 \]
Effectively, if $z\in \z$ is any non-zero element, then we can define
\[
 d\rho_{\sigma}(\sigma(z)):=[d\rho(z)]^{-1}
\]
and extend linearly to define a representation of $\z_{\sigma}$ on $V$ that, by hypothesis, is integrable. The converse implication is immediate.

$(1)\Rightarrow (2).$ Assuming condition $(1'),$ we notice that the representation $(\rho_{\sigma},V)$ of $Z_{\sigma}$ induces a differentiable action of $\Sc(\hat{Z}_{\sigma})$ on $V$ satisfying
\[
 \rho_{\sigma}(\sigma(z)\sigma(\phi))=d\rho_{\sigma}(\sigma(z))\rho_{\sigma}(\sigma(\phi)),
\]
for $\phi\in \Sc(\hat{Z}^{\times})$. On the other hand, if we define
\[
 [\sigma\rho](\psi)=\rho(\sigma^{-1}(\psi))
\]
for $\psi\in \Sc(\hat{Z}_{\sigma}^{\times})$, then according to equation (\ref{eq:sigma_z_action}), we have that
\begin{align*}
 [\sigma\rho](\sigma(z)\sigma(\phi)) & = \rho(\sigma^{-1}(\sigma(z^{-1}\phi)))  = \rho(\z^{-1}\phi)  = [d\rho(z)]^{-1}\rho(\phi)  =d\rho_{\sigma}(\sigma(z))[\sigma\rho](\sigma(\phi))
\end{align*}
for all $\phi\in \Sc(\hat{Z}^{\times})$. It follows that
\[
[\sigma\rho](\exp t\sigma(z)\cdot \sigma(\phi))=\rho_{\sigma}(\exp t\sigma(z))[\sigma\rho](\sigma(\phi))
\]
In other words, the map
\[
 \sigma\rho:\Sc(\hat{Z}_{\sigma}^{\times})\longrightarrow \End(V)
\]
intertwines the natural action of $Z_{\sigma}$ on $\Sc(\hat{Z}_{\sigma}^{\times})$ with the (left) action of $Z_{\sigma}$ on $\End(V)$ induced by $\rho_{\sigma}.$ It follows that $\sigma\rho$ would then also intertwine the corresponding actions of $\Sc(\hat{Z}_{\sigma}),$ that is, for all $\psi\in \Sc(\hat{Z}_{\sigma})$ we would have that
\begin{align}[\sigma\rho](\psi\sigma(\phi)) & =\rho_{\sigma}[\psi](\sigma(\phi)) \nonumber \\
        \rho(\sigma(\psi)\phi) & = \rho_{\sigma}(\psi)\rho(\phi) \label{eq:sigma_rho_action}
\end{align}
Similarly, if we set
\[
 [\sigma^{-1}\rho_{\sigma}](\phi)=\rho_{\sigma}(\sigma(\phi))\qquad \mbox{for $\phi\in \Sc(\hat{Z}_{\sigma}^{\times})$,}
\]
then it's straightforward to check that
\[
 [\sigma^{-1}\rho_{\sigma}](\sigma^{-1}(\psi)z)=[\sigma^{-1}\rho_{\sigma}](\sigma^{-1}(\psi))d\rho(z)
\]
for all $\psi\in \Sc(\hat{Z}_{\sigma}^{\times})$ and hence, similarly to equation
(\ref{eq:sigma_rho_action}), we have that

\begin{align}
[\sigma^{-1}\rho_{\sigma}](\sigma^{-1}(\psi)\phi)&=[\sigma^{-1}\rho_{\sigma}](\sigma^{-1}(\psi))\rho(\phi) \nonumber \\
        \rho_{\sigma}(\psi\sigma(\phi)) & = \rho_{\sigma}(\psi)\rho(\phi) \label{eq:sigma_rho_action_2}
\end{align}
for all $\psi\in \Sc(\hat{Z}_{\sigma}^{\times})$. It follows from equations (\ref{eq:sigma_rho_action}) and (\ref{eq:sigma_rho_action_2}) that for all $\phi \in \Sc(\hat{Z}^{\times}),$ $\psi \in \Sc(\hat{Z}_{\sigma}^{\times})$ we have that
\[
 [\sigma\rho](\psi\sigma(\phi))=\rho_{\sigma}(\psi)\rho(\phi)=\rho_{\sigma}(\psi\sigma(\phi)).
\]
But now, since $\Sc(\hat{Z}_{\sigma}^{\times})$ is a factorizable Fr\'echet algebra, we have that, for all $\psi\in \Sc(\hat{Z}_{\sigma}^{\times})$,
\[
 \rho_{\sigma}(\psi)=[\sigma\rho](\psi)=\rho(\sigma^{-1}(\psi)).
\]
Using this equation, and the isomorphism $\Sc(P^{1}(\hat{Z})) \simeq \Sc(\hat{Z})+_{\Sc(\hat{Z}^{\times})} \Sc(\hat{Z}_{\sigma}),$ we can define a linear map
\[
 \overline{\rho}: \Sc(P^{1}(\hat{Z})) \longrightarrow \End(V)
\]
by setting
\begin{equation}\label{eq:rho_bar_definition}
 \overline{\rho}(\phi+\psi)=\rho(\phi)+\rho_{\sigma}(\psi) \qquad \mbox{for $\phi\in \Sc(\hat{Z}),$ $\psi\in \Sc(\hat{Z}_{\sigma})$.}
\end{equation}
We can see that this map is well defined because if $\psi=\sigma(\phi),$ for some $\phi\in \Sc(\hat{Z}^{\times}),$ then
\[
 \rho_{\sigma}(\psi)=\rho_{\sigma}(\sigma(\phi))= \rho(\phi),
\]
that is, using the definition given in equation (\ref{eq:rho_bar_definition}), we have that
\[
 \overline{\rho}(\phi)=\rho(\phi)=\rho_{\sigma}(\psi)=\overline{\rho}(\psi).
\]
We now want to show that $\overline{\rho}$ actually defines a ring homomorphism. Notice that, from the definition of $\overline{\rho}$ the only thing that we need to check is that
\begin{equation}\label{eq:check_ring_homomorphism}
 \overline{\phi} \cdot \overline{\psi}=\overline{\phi\cdot \psi} \qquad \mbox{for $\phi\in \Sc(\hat{Z}),$ $\psi\in \Sc(\hat{Z}_{\sigma})$,}
\end{equation}
where in the right hand side of the above equation we are taking the product in $\Sc(P^{1}(\hat{Z}))$ and hence $\phi \cdot \psi\in \Sc(\hat{Z}^{\times})\subset \Sc(P^{1}(\hat{Z}))$. Effectively, equations (\ref{eq:sigma_rho_action}) and (\ref{eq:sigma_rho_action_2}) imply that for $\phi\in \Sc(\hat{Z})$ and $\psi\in \Sc(\hat{Z}_{\sigma}^{\times})$ or for $\phi \in \Sc(\hat{Z}^{\times})$ and $\psi \in \Sc(\hat{Z}_{\sigma})$ we already have that $\overline{\rho}(\phi)\overline{\rho}(\psi)=\overline{\rho}(fh).$ In particular, if $\psi_{0}\in \Sc(\hat{Z}_{\sigma}^{\times})$, then
\begin{align*}
 \overline{\rho}(\phi)\overline{\rho}(\psi)\overline{\rho}(\psi_{0})  =\overline{\rho}(\psi\psi_{0})  = \overline{\rho}(\phi\psi\psi_{0})
    = \overline{\rho}(\phi\psi)\overline{\rho}(\psi_{0})
\end{align*}
where in the first step we are using that $\psi\psi_{0}\in \Sc(\hat{Z}_{\sigma}^{\times})$ and in the second step we are using that $\phi\psi\in \Sc(\hat{Z}^{\times}).$ It follows that for all $v\in \rho(\Sc(\hat{Z}^{\times}) V=\rho_{\sigma}(\Sc(\hat{Z}_{\sigma}^{\times}) V$
\[
 \overline{\rho}(\phi)\overline{\rho}(\psi) v = \overline{\rho}(\phi\psi)v
\]
and, by continuity, the same holds for all $v \in \overline{\rho(\Sc(\hat{Z}^{\times}) V}\subset V.$ Therefore, if we can show that $\overline{\rho(\Sc(\hat{Z}^{\times}) V}=V$ then we would have succeed in proving equation (\ref{eq:check_ring_homomorphism}) and we would, thus, have a representation $\overline{\rho}$ of $\Sc(P^{1}(\hat{Z}))$ extending $\rho$ and such that $\overline{\rho}|_{\Sc(\hat{Z}_{\sigma})}=\rho_{\sigma}$ is differentiable. (Because $\rho_{\sigma}$ was induced from a differentiable action of $Z_{\sigma}$ on $V.$) Now, from Theorem \ref{thm:U}, we have
\[
 \big[V/\rho(\Sc(\hat{Z}_{\sigma}^{\times}))V\big]'=\{\lambda\in V' \, | \, \mbox{$[d\rho(z)]^{k}\lambda =0\}$ for some $k\in \N$.}
\]
On the other hand, since $d\rho(z)$ is invertible, we have that if $[d\rho(z)]^{k}\lambda =0,$ then $\lambda = 0.$ That is,
\[
 \big[V/\rho(\Sc(\hat{Z}_{\sigma}^{\times}))V\big]' = 0.
\]
In other words, $\rho(\Sc(\hat{Z}_{\sigma}^{\times}))V$ is dense in $V,$ as we wanted to show.

$(2)\Rightarrow (3).$ Now let us assume that $(\rho,V)$ extends to a representation $\overline{\rho}$ of $\Sc(P^{1}(\hat{Z}))$ and that $V$ is differentiable as a $\Sc(\hat{Z}_{\sigma})$-module. Then, since $V$ was already differentiable for $\Sc(\hat{Z})$ and since this algebra and $\Sc(\hat{Z}_{\sigma})$ are factorizable, it follows that
\[
 V=\Sc(\hat{Z}_{\sigma})V=\Sc(\hat{Z}_{\sigma})\Sc(\hat{Z})V=\Sc(\hat{Z}^{\times}) V,
\]
that is, $V$ is differentiable for $\Sc(\hat{Z}^{\times})$ as we wanted to show.

$(3)\Rightarrow (1).$ Now assume that $V$ is differentiable as an $\Sc(\hat{Z}^{\times})$-module. Then, since the action of $z$ on $\Sc(\hat{Z}^{\times})$ is invertible for all $z\in \z$, it follows that the operator $d\rho(z)$ is invertible with inverse given by
\[
 [d\rho(z)]^{-1}\rho(\phi)v=\rho(z^{-1}\phi)v \qquad \mbox{for $\phi\in \Sc(\hat{Z}^{\times})$, $v\in V.$}
\]
Now define an action of $\Sc(\hat{Z}_{\sigma}^{\times})$ on $V$ by setting $\rho_{\sigma}(\psi)=\rho(\sigma^{-1}(\psi))$ for $\psi \in \Sc(\hat{Z}_{\sigma}^{\times}).$ Since $\Sc(\hat{Z}_{\sigma}^{\times})$ is an $\Sc(\hat{Z}_{\sigma})$-algebra with  an approximate identity and since $V$ is clearly differentiable for $\Sc(\hat{Z}_{\sigma}^{\times})$, it follows from Lemma \ref{lemma:B_A-algebra_with_approximate_identity} that $\rho_{\sigma}$ extends to an  action of $\Sc(\hat{Z}_{\sigma})$. 
\DimaA{Clearly this action is differentiable, 
 and hence defines a smooth representation of $Z_{\sigma}$.} Furthermore,
\begin{multline*}
 d\rho_{\sigma}(\sigma(z))\rho_{\sigma}(\sigma(\phi))  =\rho_{\sigma}(\sigma(z)\sigma(\phi))=\\      = \rho_{\sigma}(\sigma(z^{-1}\phi))  = \rho(z^{-1}\phi)  = [d\rho(z)]^{-1}\rho(\phi) = [d\rho(z)]^{-1}\rho_{\sigma}(\sigma(\phi))
\end{multline*}
for $\phi\in \Sc(\hat{Z}^{\times})$ and $z\in \z^{\times}.$ Using again that $V$ is differentiable for $\Sc(\hat{Z}^{\times})$, it follows that
\[
 d\rho_{\sigma}(\sigma(z))=[d\rho(z)]^{-1}
\]
as we wanted.
\end{proof}

The following example shows that it is not automatic that a map like the one given in equation (\ref{eq:rho_bar_definition}) should necessarily define an algebra homomorphism.


\begin{example}
 The module
 \[
  V=\Sc(\hat{Z})\otimes_{\Sc(\hat{Z}^{\times})}\Sc(\hat{Z}_{\sigma})
 \]
satisfies the following properties:
\begin{enumerate}
 \item If $(\rho,V)$ and $(\rho_{\sigma},v)$ denote the natural actions of $Z$ and $Z_{\sigma}$ on $V$, then $V$ is differentiable for $\Sc(\hat{Z})$ and $\Sc(\hat{Z}_{\sigma}).$
 \item For $\psi\in \Sc(\hat{Z}_{\sigma}^{\times}),$
 \[
  \rho_{\sigma}(\psi)=\rho(\sigma^{-1}(\psi)).
 \]

 \item If $\phi\in \Sc(\hat{Z})$ and $\psi\in \Sc(\hat{Z}_{\sigma}^{\times})$ or $\phi\in \Sc(\hat{Z}^{\times})$ and $\psi\in \Sc(\hat{Z}_{\sigma}),$ then
 \[
  \rho(\phi)\rho_{\sigma}(\psi)=\rho(fh).
 \]
\item If $\phi_{0}\in \Sc(\hat{Z})$ and $\psi_{0}\in \Sc(\hat{Z}_{\sigma})$ satisfy $\phi_{0}(0)=1$ and $\psi_{0}(0)=1,$ then
\[
 \rho(\phi_{0})\rho_{\sigma}(\psi_{0})\neq \rho(\phi_{0}\psi_{0}).
\]
Effectively, let $T$ be the linear functional on $V$ induced by
\[
 T(\phi\otimes \psi)=\phi(0)\psi(0).
\]
Then, for $\phi_{0}$ and $\psi_{0}$ as above it's clear that
\[
 \rho(\phi_{0})\rho_{\sigma}(\psi_{0}) \cdot T = T
\]
but
\[
  \rho(\phi_{0}\psi_{0}) \cdot T =0.
\]

\end{enumerate}

\end{example}

\begin{defn}
Denote the category of non-singular representations of $Z$ by $\Rep^{\infty}(Z)^{\times}$.
\end{defn}
A fundamental example of such representation is $\varpi\cong\Sc(\hat{Z}^{\times})$.
Let $\overline{\varpi}$ denote the complex conjugate of $\varpi$. It is also isomorphic to the dual of $\varpi$, since $\varpi$ has a canonical scalar product given by $\int_{\hat{Z}^{\times}} fg d\chi$ for any Haar measure $d\chi$ on $\hat{Z}^{\times}\simeq \R^{\times}$.
From Lemma \ref{lem:nonsin} we obtain the following corollary.

\begin{cor}\label{cor:w}
For any $\rho\in \Rep^{\infty}(Z)^{\times}$, the action map $\Sc(\hat{Z}^{\times})\hot \rho\to \rho$ defines an isomorphism $\overline{\varpi}\hot_Z\rho\to \rho$.
\end{cor}

\begin{example}
For any non-trivial character $\chi\in \widehat{Z}^{\times}$, let $\varpi_{(\chi)}$ denote the quotient of $\varpi$ by the submodule of functions flat at $\chi$. Then $\varpi_{(\chi)}\in \Rep^{\infty}(Z)^{\times}$. Note that it is the representation of formal power series at $\chi$, and that it is indecomposable. \end{example}
It is easy to see that the only irreducible representations in $\Rep^{\infty}(Z)^{\times}$ are the one-dimensional ones.
%
%
%
%

\section{Non-singular representations of the Heisenberg group}\label{sec:cat}

Let $(W,\omega)$ be a symplectic vector space, and $H=W\ltimes Z$ be the corresponding Heisenberg group with center $Z$.

\begin{definition}
We say that a representation $(\rho, V)$ of $H$ is \emph{non-singular} if it is non-singular for $Z$.
Denote the category of non-singular representations of $H$ by $\Rep^{\infty}(H)^{\times}$.
\end{definition}
In this section we establish an equivalence between $\Rep^{\infty}(H)^{\times}$ and the category  $\Rep^{\infty}(Z)^{\times}$ of non-singular representation of $Z$. We view this as a generalization of the Stone-von-Neumann theorem.

Recall that   $\varpi\in \Rep^{\infty}(Z)^\times$ denotes the representation of $Z$ in the algebra $\Sc(\widehat{Z}^{\times})$ of Schwartz functions on the space $\widehat{Z}^{\times}$ of non-trivial unitary characters of $Z$, given by $\varpi(z)\phi(\chi):=\chi(z)\phi(\chi)$.

Choose a Lagrangian subspace $L\subset W$. Then $L\times Z$ is a maximal abelian subgroup of $H$.
\DimaA{Define a representation $\Omega \in \Rep^{\infty}(H)^\times$ by $\Omega:=\ind_{L\times Z}^H (\varpi),$ where we let $L$ act trivially on $\varpi$.}

\begin{prop}\label{prop:LagIndep}
Let $R\subset W$ be another Lagrangian subspace. Then $$\ind_{L\times Z}^H\varpi\simeq \ind_{R\times Z}^H\varpi.$$
\end{prop}
\begin{proof}
We have an isomorphism  given by
\begin{equation}\label{=RootChange}
f\mapsto \check f, \quad \check f(g)=\int_{r\in R/(L\cap R)}f(rg)dr,
\end{equation}
where $f$ and $\check f$ are $\varpi$-valued functions on $H$.
\DimaA{For a canonical choice of the Lebesgue measure $dr$ see \cite[\S 1.4]{LV}.
The inverse map is given by 
\begin{equation}\label{=RootChangeBack}
f\mapsto \hat f, \quad \hat f(g)=\int_{l\in L/(L\cap R)}f(lg)dl,
\end{equation}
To see that both compositions of the two operators are identities, let $f\in \ind_{L\times Z}^H \varpi$. We have to show that $\widehat{ \check f}=f$. Equivalently, for any $\chi\in \widehat{Z}^{\times}$ we have to show 
$\widehat{ \check f}(\chi)=f(\chi)$. This is a version of the Fourier inversion formula, proven in \cite[\S 1.4]{LV}. 
In the same way one shows that the other composition is \Raul{the} identity.
}
\end{proof}
\begin{remark}
By the definition of Schwartz induction, we obtain natural projections $p_L:\Sc(H,\varpi)\to \ind_{L\times Z}^H \varpi$ and $p_R:\Sc(H,\varpi) \to \ind_{R\times Z}^H \varpi$, so it is natural to ask whether these projections are related by the isomorphism between $\ind_{L\times Z}^H \varpi$ and $\ind_{R\times Z}^H \varpi$ given above. This is, however, not the case and
\[
 \widecheck {p_L(f)}\neq p_R(f)
\]
in general.
\end{remark}


\begin{proof}[Proof of Theorem \ref{thmB}]

By Lemma \ref{lem:Frob} and Corollary \ref{cor:w}, the functor $V\mapsto V\hot_H \overline{\Omega}$ is isomorphic to the functor $V\mapsto V_L$, while the functor $V_0\mapsto \Raul{V_0}\hot_Z\Omega$ is isomorphic to the functor given by $V_0\mapsto \ind_{L\times Z}^HV_0$, where we let $L$ act on $V_0$ trivially.
We thus have to show that the two functors are mutually quasi-inverse.

Let $V\in \Rep^{\infty}(H)^{\times}$ and restrict $V$ to $L\times Z$. Since $L\times Z$ is abelian, we can use the Fourier transform to obtain an $\Sc(\widehat{L}\times \widehat{Z})$-module structure on $V$. Now restrict this structure to $\Sc(\widehat{L}\times \widehat{Z}^{\times})$ and notice that, since $V\in\Rep^{\infty}(H)^{\times}$, then $V$ is actually a differentiable $\Sc(\widehat{L}\times \widehat{Z}^{\times})$-module. Now, using the action of $H$ on $V$, we can construct a differentiable imprimitivity system $\widetilde{V}\in \Sc\mod_{H,\widehat{L}\times \widehat{Z}^{\times}}$, where the action of $H$ on $\widehat{L}\times \widehat{Z}$ is given by
$$((y,z_1)\cdot(\psi,z'))(x,z_2)=\psi(x)+z'(\omega(x,y))+z'(z_2),$$
where $x,y\in W; \, z_1,z_2\in Z; \, \psi\in \widehat{L},$ and $z'\in \widehat{Z}^{\times}$.
Identify $\widehat{Z}^{\times}$ with $\R^{\times}$ and define an automorphism of  $\widehat{L}\times \widehat{Z}$ by $(\psi,z')\mapsto (z'\psi,z')$. After this automorphism, the action of $H$ becomes the product of a transitive action on $\widehat{L}$ and the trivial action on $\widehat{Z}$. The stabilizer of every point is $L\times Z$. Thus, by Theorem \ref{thm:imprimitivity}, taking the fiber at $0\in \widehat{L}$ defines an equivalence of categories  between $\Sc\mod_{H,\widehat{L}\times \widehat{Z}^{\times}}$ and
$\Sc\mod_{L\times Z,\widehat{Z}^{\times}}$. The fiber at zero is isomorphic to the space $V_L$ of $L$-coinvariants.

We note that the action of  $\Sc(\widehat{L}\times \widehat{Z}^{\times})$ on $V$ is given by the action of $H$. Moreover,
the functor $V\mapsto \widetilde{V}$ defines an equivalence of categories between $\Rep^{\infty}(H)^{\times}$ and the subcategory of $\Sc\mod_{H,\widehat{L}\times \widehat{Z}^{\times}}$ consisting of modules on which the action of  $\Sc(\widehat{L}\times \widehat{Z}^{\times})$
is given by the action of $H$. This category is in turn equivalent to the subcategory of $\Sc\mod_{L\times Z,\widehat{Z}^{\times}}$ consisting of modules on which $L$ acts trivially, and the action of $\Sc(\widehat{Z}^{\times})$ is given by the action of $Z$ (and is differentiable). This category in turn is equivalent to $\Rep^{\infty}(Z)^{\times}$.
\end{proof}

\begin{remark}\label{rem:NatTran}
The natural transformations $CI\to \Id$ and $IC\to \Id$ constructed in the proof can be also defined using the canonical map $\Omega\hot_H\overline{\Omega}\to \varpi$ that we will now describe.
Fix a Lebesgue measure $dw$ on $W/L$.
Let $f\in \Omega =\ind_{L\times Z}^H\varpi,$ and $g\in \overline{\Omega} =\ind_{L\times Z}^H\overline{\varpi}$. Then $f$ and $g$ are $L\times Z$-equivariant functions from $H$ to $\varpi=\Sc(\widehat{Z}^{\times})$. Using the algebra structure on $\varpi$, we can consider the product $fg$ as a function $fg:H\to \varpi$. This function is $L\times Z$-invariant, and thus can be viewed as a function on $H/(L\times Z)=W/L$. We map $f\otimes g\mapsto \int_{W/L}fgdw$.
This map intertwines the actions of $Z$, where $Z$ acts on $\Omega\hot_H\overline{\Omega}\to \varpi$  through acting on $\Omega$.

Passing to $(Z,\chi)$-coinvariants (for any $\chi \in \widehat{Z}^{\times})$ we obtain from this construction the classical scalar product on the oscillator representation $\Omega_{\chi}$.

\end{remark}

\section{Generalized Weil representation}\label{sec:Weil}
 In this section we define and study the Weil representation of  $\widetilde{\Sp}(W)$ on
$\Omega$.
To define it  is the same as to define a group homomorphism $\Sp(W)\to \Aut_{\C}(\Omega)/\{\pm\Id\}$, where $\Aut_{\C}(\Omega)$ denotes the group of all continuous invertible linear operators on the space of $\Omega$. \Raul{To write this group homomorphism, we will realize $\Omega=\ind_{L\times Z}^H \varpi$ as the space of $\varpi$-valued Schwartz functions on $R$, where $R$ is a complementary Lagrangian to $L$ in $W.$ }


Let $n:=\dim W/2$. Introduce a basis on $W$ such that $L$ is spanned by the last $n$ basis vectors, and the symplectic form is given in this basis by
\[
 J=\begin{pmatrix}0 & \Id \\
-\Id & 0 \\
\end{pmatrix}.
\]
 We identify the   Lagrangian spanned by the first $n$ coordinates with $R\cong W/L$. To define an analogue of Fourier transform we first  define  a special operator $S$ on the space of $\varpi$ by
\begin{equation}
(S\phi)(\chi_s):=\sqrt{ |s|}\phi(\chi_s),\text{ where }\chi_s(t):=\exp(2\pi i st)
\end{equation}
We now define
analogues of Fourier transform and its inverse on the space of $\Omega$ by
\begin{equation}
\cF_{\varpi}(f)(x):=S^{-1}\int_{L}\varpi(-x^ty)f(y)dy \qquad \text{ and } \qquad \cF_{\varpi}^{-1}(f)(x):=S\int_{L}\varpi(x^ty)f(y)dy.
\end{equation}

Define the homomorphism $\Sp(W)\to \Aut_{\C}(\Omega)/\{\pm\Id\}$ on the generators of  $\Sp(W)$ by
\begin{eqnarray}\label{=Weil}
\sigma\begin{pmatrix}A & 0 \\
0 & (A^*)^{-1} \\
\end{pmatrix}:&f(x)\mapsto\pm(\det A)^{-1/2}f(A^{-1}x)\\ \label{=Weil2}
\sigma\begin{pmatrix}\Id & 0 \\
C & \Id \\
\end{pmatrix}:&f(x)\mapsto\pm\varpi(-x^{t}Cx/2)f(x)\\ \label{=Weil3}
\sigma\begin{pmatrix}0 & \Id \\
-\Id & 0 \\
\end{pmatrix}:&f(x)\mapsto\pm i^{n/2}\cF_{\varpi}^{-1}(f)(x)
\end{eqnarray}
%
%

\begin{lemma}\label{lem:Weil}
The formulas \eqref{=Weil}-\eqref{=Weil3} define a representation of $\widetilde{\Sp}(W)$ on $\Omega$. Moreover, this representation, together with the action of $H$, defines a representation of $\widetilde{\Sp}(W)\ltimes H$.
\end{lemma}


\begin{proof}
For the proof  we realize $\Omega$ in the space $\Sc(W/L\times \widehat{Z}^{\times})$, and further identify $\widehat{Z}^{\times}$ with $\R^{\times}$ by $\chi_s\leftrightarrow s$.
In these terms the formulas \eqref{=Weil}-\eqref{=Weil3} become
\begin{eqnarray}\label{=WeilE}
\sigma\begin{pmatrix}A & 0 \\
0 & (A^*)^{-1} \\
\end{pmatrix}:&F(x,s)\mapsto\pm(\det A)^{-1/2}F(A^{-1}x,s)\\ \label{=WeilE2}
\sigma\begin{pmatrix}\Id & 0 \\
C & \Id \\
\end{pmatrix}:&F(x,s)\mapsto\pm\chi_s(x^{t}Cx/2)F(x,s)\\ \label{=WeilE3}
\sigma\begin{pmatrix}0 & \Id \\
-\Id & 0 \\
\end{pmatrix}:&F(x,s)\mapsto\pm i^{n/2}\cF_{s}^{-1}(F(x,s)),
\end{eqnarray}
where
$$   \cF_{s}^{-1}(f)(x):=\sqrt{|s|}\int_{L}\chi_s(x^ty)f(y)dy.$$
 The statement that these formulas define a representation amounts then to verifying some identities for any  $F(x,s) \in \Sc(W/L\times \R^{\times})$.
Fix $s\in \R^{\times}$.
Conjugating the operators in \eqref{=WeilE}-\eqref{=WeilE3} by the dilation $f(x) \mapsto f(\sqrt{|s|}x)$, we can reduce to the case $s=\pm 1$. In this case the representation becomes one of the two classical Weil representations (see {\it e.g.} \cite[(4.24)-(4.26) and Proposition 4.46]{Fol}). In the same way one verifies that these formulas together with the action of $H$ define a representation of $\widetilde{\Sp}(W)\ltimes H$.
\end{proof}

\begin{proof}[Proof of Proposition \ref{prop:OmExt}]
We just proved in Lemma \ref{lem:Weil} the existence of an extension of the action of $H$ on $\Omega$ to an action of $\widetilde{\Sp}(W)\ltimes H$, and it is left to prove the uniqueness of such an extension $\pi$.
Since $\widetilde{\Sp}(W)$ commutes with $Z$ inside $\widetilde{\Sp}(W)\ltimes H$, $\pi$ preserves the kernel of the projection $\Omega\to \Omega_{\chi}$, and thus defines an action $\pi_{\chi}$ \DimaA{on $\Omega_\chi$}. We can realize $\Omega$ in $\Sc(W/L\times\widehat{Z}^{\times})$, and $\Omega_{\chi}$ in $\Sc(W/L)$. Under these identifications, the projection $\Omega\to \Omega_{\chi}$ is evaluation \DimaA{at $\chi\in \widehat{Z}^{\times} $}. Thus, every element of $\Omega$ is determined by its images in all ${\chi}$, and therefore $\pi$ is determined by the collection of all $\pi_{\chi}$. Thus it is enough to prove the uniqueness of $\pi_{\chi}$. Since $\Mp(W)$ has no characters, it is enough to prove uniqueness as a projective representation. That is, we have to prove that for every $g\in \Sp(W)$, there is at most a one-dimensional space of isomorphisms between $\Omega_{\chi}$ and $\Omega_{\chi}^g$ as representations of $H$, where $\Omega_{\chi}^g$ is $\Omega_{\chi}$ with the action of $H$ twisted by the \Raul{automorphism} of $H$ defined by $g$. This, in turn, follows from  $I$ being an equivalence \DimaA{of categories} and $\chi$ being one-dimensional, since $\Omega_{\chi}=I(\chi)$.
\end{proof}

\begin{proof}[Proof of Theorem \ref{thm:FJ}]
Let us first prove that $C^+I^+\cong \Id$. To prove this we have to prove that for any $\rho \in \Rep^{\infty}(\Sp\times Z)^{\times}$, the isomorphism $\alp_{\rho}:\rho\tilde\to  \rho\hot_Z\Omega\hot_H\Omega$ as representations of $Z$ constructed in the proof of Theorem \ref{thmB} commutes with the action of $\Mp$.
We will first prove this for special cases of $\rho$ and then deduce the general case.
\begin{enumerate}[{Case} 1]
\item $\rho$ is one-dimensional. \\
In this case $\rho\cong \rho\hot_Z\Omega\hot_H\overline{\Omega}$ is also one-dimensional, thus the action of $\Mp$ on both is trivial, and commutes with any operator.
\item $\rho \cong \varpi$. \\
Recall that the \Raul{underlying} space of $\varpi$ is $\Sc(\widehat{Z}^{\times})$. Thus we have to show that for any $\chi\in\widehat{Z}^{\times}$ and any $g\in \Mp$, $\alp_{\varpi}(\phi^g)(\chi)=(\alp_{\varpi}(\phi))^g(\chi)$.

Since the evaluation at $\chi$, as a map $\varpi\to \chi$, is a morphism in $\Rep^{\infty}(\widehat{Z})^{\times}$, and
 $\alp$ is a natural transformation, $\alp$ commutes with the evaluation at $\chi$. Thus $\alp_{\varpi}(\phi^g)(\chi)=(\alp_{\varpi}(\phi))^g(\chi)$ by the previous case.

\item $\rho \cong \varpi\hot_{\C}V$, where $Z$ acts only on $\varpi$.\\
This case follows from the previous one.\
\item General\\
This case again follows from the previous one, since any $\rho \in  \Rep^{\infty}(\widehat{Z})^{\times}$ is a quotient of a representation of the form $\varpi\hot_{\C}V$, and $\alp:CI\cong \Id$ is a natural transformation.
\end{enumerate}

Let us now prove $I^+C^+\cong \Id$ using a similar argument. The isomorphism $IC(\Omega)\cong \Omega$ intertwines the actions of $\Mp$ since by Proposition \ref{prop:OmExt} there is only one action of $\Mp$ on the space of $\Omega$ that is compatible with the action of $H$. This implies the statement in general, since every $\tau \in  \Rep^{\infty}(\widehat{Z})^{\times}$ is a quotient of a representation of the form $\Omega\hot_{\C}V$, and the isomorphism $IC\cong \Id$ is a natural transformation.
\end{proof}

\begin{remark}
\begin{enumerate}[(i)]
\item One can also show explicitly that the canonical map $\Omega\hot_H\overline{\Omega}\to \varpi$ is invariant under the diagonal action of the symplectic group, and deduce Theorem \ref{thm:FJ} from this fact, cf. Remark \ref{rem:NatTran} above.
\item The constructed action of $\Mp(W)$\ on $C(\rho)$ for any $\rho \in\Rep^{\infty}(Z)^{\times}$
is given by formulas completely analogous to \eqref{=Weil}-\eqref{=Weil3}, but with $\rho$ in place of $\varpi$.
\end{enumerate}
\end{remark}

\begin{proof}[Proof of Theorem \ref{thm:Ext}]
By Theorem \ref{thm:FJ}, it is enough to show that the restriction functor $\Rep^{\infty}(\Sp\times Z)^{\times}\to \Rep^{\infty}( Z)^{\times}$ has a unique quasi-inverse. This  follows from the argument in the proof of Theorem \ref{thm:FJ}.
\end{proof}

Let us now demonstrate formulas \eqref{=Weil}-\eqref{=Weil3} and Lemma \ref{lem:Weil}
in several examples. In the first two  examples we let $\rho$  be the Taylor expansion by $s$ of  $\exp(2\pi i s t)$ at $s=1$, of orders 2 and 3 respectively. We also let $W$ be two-dimensional.

\begin{exm}
$\rho$ is two-dimensional, with $\rho(t)=\begin{pmatrix}\exp(2\pi i t) & 2\pi it\exp(2\pi i t) \\
0 & \exp(2\pi i t) \\
\end{pmatrix}$
\end{exm}
Let us verify the Fourier inversion formula $\sigma(\begin{pmatrix}0 & \Id \\
-\Id & 0 \\
\end{pmatrix})^2f(x)= if(-x)$ for this example by a direct computation.
The action of $S$ in this case is $\begin{pmatrix}1 & 1/2 \\
0 & 1 \\
\end{pmatrix}$
and $\sigma(\begin{pmatrix}0 & \Id \\
-\Id & 0 \\
\end{pmatrix})$ is given by $\pm\begin{pmatrix}1 & E \\
0 & 1 \\
\end{pmatrix}\cF^{-1},$
where $E=(x\partial+\partial x)/2=x\partial+1/2$ is the symmetrized Euler operator, and $\cF^{-1}=\cF_{1}^{-1}$ is the classical inverse Fourier transform. Since $\cF^{-1}$ conjugates $E$ to $-E$, we have
$(\begin{pmatrix}1 & E \\
0 & 1 \\
\end{pmatrix}\cF^{-1})^2=i\cF^{-2}\Id$, and the statement follows now from the classical Fourier inversion formula.

\begin{exm}
$\rho$ is 3-dimensional, with
$\rho(t)=\begin{pmatrix}\exp(2\pi i t) & 2\pi it\exp(2\pi i t) & -2\pi^2t^2\exp(2\pi i t)\\
0 & \exp(2\pi i t) & 2\pi it\exp(2\pi i t)\\
0 & 0 & \exp(2\pi i t)
\end{pmatrix}$
\end{exm}
\begin{proof}[Proof of the Fourier inversion]
The action of $S$ in this case is $\begin{pmatrix}1 & 1/2 & -1/4\\
0 & 1 & 1/2 \\
0 & 0 & 1
\end{pmatrix}$.
Also, in this case
$\sigma(\begin{pmatrix}0 & \Id \\
-\Id & 0 \\
\end{pmatrix})=\pm i^{1/2}\begin{pmatrix}1 & E & E^2/2-E/2\\
0 & 1 & E \\
0 & 0 & 1
\end{pmatrix}\cF^{-1}$.
It is easy to see that this squares to $i\cF^{-2}\Id$.
\end{proof}

\begin{exm}
 Squaring the operator in \eqref{=WeilE3}  we get $F(x,s)\mapsto iF(-x,s)$. \Raul{Taking the derivative of} this identity in the variable $s$, we get generalizations of the previous examples to any order. Let us now show that in this example the formulas \eqref{=WeilE}-\eqref{=WeilE3} define a representation of the Lie algebra $\sll_2$. For this, we will conjugate the action of the lower-triangular nilpotent element $Y$ given by differentiating \eqref{=WeilE2} by the operator \Raul{given} in  \eqref{=WeilE3}, thus obtaining the action of the upper-triangular nilpotent $X$. Differentiating \eqref{=WeilE} we get the action of the diagonal element $H\in \sll_2$, and then check the standard $\sll_2$ relations.

\Raul{Differentiating} \eqref{=Weil2}, we have $$\sigma(Y)F(x,s)=\frac{\partial}{\partial c}\exp(-2\pi i s cx^2/2)F(x,s)|_{c=0}=-\pi isx^2F(x,s).$$
\Raul{Differentiating}  \eqref{=Weil}, we have $$\sigma(H)F(x,s)=\frac{\partial}{\partial c}\exp(-c/2)F(\exp(-c)x,s)|_{c=0}=-F(x,s)/2-x\partial_xF(x,s)=-EF(x,s).$$
Conjugating the action of $Y$ by  $\sigma(\begin{pmatrix}0 & \Id \\
-\Id & 0 \\
\end{pmatrix})$ we get
$$\sigma(X)(F)(x,s)=-\pi i s\cF(x^2\cF^{-1}(F(x,s)))(-xs),s)=(4\pi i s)^{-1}\partial_x^2F(x,s).$$
Since $[E,x^2]=2x^2, [E,\partial_x^2]=-2\partial_x^2,$ and $[\partial_x^2,x^2]=4E$, we get that $\sigma(X),\sigma(H),\sigma(Y)$ form an $\sll_2$-triple. If we substitute $s:=1$ we obtain the formulas for the classical infinitesimal Weil representation, see \cite[(4.43)]{Fol}.
\end{exm}


Let us now generalize another classical fact, saying that the Weil representation  only depends on the ``square class'' of the central character.
\begin{lemma}
Assume that $\rho$ is only supported on one square class. Then, as a representation of $\widetilde{\Sp}(W)$, $\sigma\cong \sigma_0\hot \rho,$ where $\sigma_0$ is the classical Weil representation corresponding to this square class, and the action of  $\widetilde{\Sp}(W)$ is on $\sigma_0$.
\end{lemma}
\begin{proof}
It is enough to prove \Raul{the result} for $\rho = A^+$, where $A^+$ is the representation of Schwartz functions on one square class $C$. Choosing a character $\chi \in C$, we identify $C$ with the multiplicative group of positive real numbers $\R_{>0}^{\times}$. We can then realize both representations in question in the Schwartz space $\Sc(H\times \R_{>0}^{\times})$, and define the intertwining operator by $$\varphi(F)(x,s):=F(\sqrt{|s|} x,s).$$
\end{proof}

\section{Generalized coinvariants}\label{sec:coinv}
%

\DimaA{Let $N$ be a unipotent real algebraic group, and $\fn$ be its Lie algebra.}
\begin{lemma}\label{lem:fin}
For any $i>0$, $\fn^i\Sc(N)$ has finite codimension in $\Sc(N)$.
\end{lemma}
\begin{proof}
For $i=1$ this is \cite[Lemma B.1.4]{AG_ST}. Thus $$\dim \fn^i\Sc(N)/\fn^{i+1}\Sc(N)\leq\dim (\fn^i\otimes (\Sc(N)/\fn\Sc(N)))<\infty.$$
This implies the general case by induction.
\end{proof}

\begin{defn}
Define $I_N:=\bigcap\fn^i\Sc(N)$ and $\Sc(N)_{(N)}:=\Sc(N)/I_N$.
\end{defn}

\begin{exm}
For the additive group $\R$, $I_{\R}$ consists of Schwartz functions whose Fourier transform is flat at zero. Thus, $\Sc(\R)_{(\R)}$ is isomorphic via Fourier transform to the space of power series $\C[[t]]$, on which each $x\in \R$ acts by multiplication by $\exp(ixt)$ .
\end{exm}
\DimaA{\begin{lem}\label{lem:func}
Any continuous functional on $\Sc(N)$ that vanishes on $I_N$ is annihilated by some power of $\fn$.
\end{lem}
\begin{proof}
One can present the space of Schwartz functions $\Sc(N)$ as an inverse limit of Banach completions $\Sc(N)_{\sigma}$, where $\sigma$ is a semi-norm on $\Sc(N)$ that involves only differential operators up to some finite order.
Thus any continuous functional on $\Sc(N)$ is continuous in one of the semi-norms $\sigma$. Thus, if it vanishes on $I_N$ it vanishes on the closure of $I_N$ with respect to $\sigma$. This closure lies in $\fn^i\Sc(N)$ for some $i$. Thus the functional is annihilated by $\fn^i$.
\end{proof}

\begin{cor}
The spaces $\Sc(N)_{(N)}$ and $\lim\limits _{\ot}\Sc(N)/\fn^i\Sc(N)$ are nuclear \Fre spaces, and the natural map $\Sc(N)_{(N)} \to \lim\limits _{\ot}\Sc(N)/\fn^i\Sc(N)$ is an isomorphism.
\end{cor}
\begin{proof}
By Lemma \ref{lem:fin}, $\fn^i\Sc(N)$ is a closed subspace of $\Sc(N)$, since continuous linear operators with finite corank between nuclear \Fre spaces have closed image (see Lemma \ref{lem:FinGood}).
Thus $I_N$ is also closed, and thus $\Sc(N)_{(N)}=\Sc(N)/I_N$\ is a nuclear \Fre space. This also implies that $\Sc(N)/\fn^i\Sc(N)$ is a \Fre space for every $i$, and thus so is their limit $\lim\limits _{\ot}\Sc(N)/\fn^i\Sc(N)$. The last statement of the lemma follows from Lemma \ref{lem:func} and the  duality equivalence between NF and DNF spaces.
\end{proof}
}
\begin{defn}
For any $V\in \DimaA{\Rep^{\infty}(N)}$, define $V_{(N)}:=V/\overline{I_NV}$.
\end{defn}

\begin{remark}
It is possible that $I_NV$ is already closed.
\end{remark}

\begin{lemma}\label{lem:GenFrobN}
$V_{(N)}\cong \Sc(N)_{(N)}\hot_NV:=(\Sc(N)_{(N)}\hot V)_N$.
\end{lemma}
\begin{proof}
We have two actions of $N$ on $\Sc(N)\hot V$: the action on $\Sc(N)$ and the diagonal action. By \cite[Lemma A.0.4]{GGS}, the two representations are isomorphic. By \cite[Corollary A.0.5]{GGS}, we have $V\cong \Sc(N)\hot_N V$, where $\hot_N$ uses the
action of $N$ on $\Sc(N)$ by right shifts. Under this isomorphism, the action of $N$ on $V$ translates to the action of $N$ on $\Sc(N)$ by left shifts.
\end{proof}

\begin{lemma}\label{lem:co*} \DimaA{For any $V\in \DimaA{\Rep^{\infty}(N)}$ we have}
$$V_{(N)}^*\cong \{\xi\in V^*\, | \, I_N\xi =0\}\cong \{\xi \in V^*\, |\, \fn^i\xi=0 \text{ for some }i\}\cong \lim_{\to} (V/\fn^iV)^* \cong (\lim_{\ot} V/\fn^iV)^*$$
\end{lemma}
\begin{proof}
First and third isomorphism are clear. For the second one, by the Baire category theorem it is enough to show that for any $v\in V$ the corresponding distribution $\xi v$  on $N$ is killed by some power of $\fn$. This follows from the previous lemma. The last isomorphism follows from the duality equivalence between NF and DNF spaces.
\end{proof}

\begin{cor}\label{cor:lim}
The map $V_{(N)}\to \lim\limits_{\ot} V/\overline{\fn^iV}$ is an \DimaA{isomorphism}.
\end{cor}

\begin{lemma}\label{lem:GenFrob}
For \DimaA{any real algebraic group $G\supset N$ and any $V\in \Rep^{\infty}(G)$}, we have $$V_{(N)}\cong V\hot_G \Sc(G)_{(N)}$$ \end{lemma}
This is proven in the same way as Lemma \ref{lem:GenFrobN}.

\begin{lemma}\label{lem:pol}
We have $\Sc^*(N)^{(\fn)}\cong \C[N],$ where $\Sc^*(N)^{(\fn)}$ denotes the space of tempered distributions on $N$ on which $\fn$ acts locally nilpotently, and $\C[N]$ denotes the space of distributions consisting of products of a Haar measure on $N$ by polynomials on $N$.
\end{lemma}
\begin{proof}
We prove this by induction on the dimension of $\fn$.
For the base case if $\dim \fn =1,$ then $\fn$ is abelian, and the statement is standard.

For general nilpotent $\fn$, let $z$ be a non-zero central element, and
$Z$ be the corresponding one-dimensional central subgroup of $N$.
Let $\lambda \in [\Sc^*(N)]^{(\fn)}$. Then there exists $m\in \N$ such that $z^{m}\lambda =0.$ We want to prove that $\lambda \in \C[N]$. We will do it by induction on $m$: If $z\lambda=0$, then we can think of $\lambda$ as a distribution on $N/Z$ and hence it would be polynomial by our induction hypothesis on $\dim N$.  Now let us assume that the result is true up to some $m,$ i.e., if $z^{m}\lambda =0,$ then $\lambda \in \C[N]$ and assume that $\lambda \in [\Sc^{*}(N)]^{(\fn)}$ satisfies the equation
\[
z^{m+1}\lambda =0
\]
But then $z^m(z\lambda)=0$, thus $z\lambda \in \C[N]$ by the induction hypothesis.
Since $z$ is central we can check (using local coordinates so that $z$ corresponds to just taking derivative with respect to a local coordinate) that there exists a polynomial $q$ such that
\[
zq=z\lambda.
\]
But then
\[
z(\lambda-q)=0
\]
from which we conclude  that $\lambda -q,$ and hence $\lambda,$ is a polynomial.
\end{proof}

From Lemmas \ref{lem:GenFrob} and \ref{lem:pol} we obtain the following corollary.

\begin{cor}\label{cor:NUPol}
Let $U$ be a unipotent \DimaA{real algebraic} group and $N\subset U$ be a normal subgroup.
Then the natural map $\C[N]\otimes \Sc^*(U/N)\to \Sc^*(U)^{(N)}$ is an isomorphism.
\end{cor}
\begin{proof}
By Lemmas \ref{lem:co*} and \ref{lem:pol} we have $(\Sc(N)_{(N)})^*\cong\Sc^*(N)^{(\fn)}\cong  \C[N]$.  Thus the map in question is dual to the map $\Sc(U)_{(N)}\to \Sc(U)\hot_N\Sc(N)_{(N)}$
which is an isomorphism by Lemma \ref{lem:GenFrob}.
\end{proof}

\DimaA{
For any unitary character $\chi$ of $N$ and any $V\in \Rep^{\infty}(N)$, we define twisted generalized co-invariants $V_{(N,\chi)}:=(V\otimes \chi^{-1})_{(N)}$. 

\begin{lem}\label{lem:Mod}
Let $G$ be a real algebraic group that includes $N$ as an algebraic subgroup.
Let $\fg$ be the Lie algebra of $G$. Suppose that 
for every $X\in \fn$, $ad(X)$ is a nilpotent operator on $\fg$. Let $\chi$ be a unitary character of $N$. Let $V\in \Rep^{\infty}(G)$. Then  $V_{(N,\chi)}$ has a natural $\fg$-module structure.
\end{lem}
\begin{proof}
Let $\fn_{\chi}:=\{X-d\chi(X)\, \vert X\in \fn\}\subset \cU(\fn)$ and let $k$ be the natural number such that $ad(X)^k=0$ for any $X\in \fn$. Then $\fg\fn_{\chi}^i\cU(\fg)\subset \fn_{\chi}^{i-k}\cU(\fg)$ for any $i>k$. Then $\fg\overline{\fn_{\chi}^iV}\subset \overline{\fn_{\chi}^{i-k}V}$ for any $i>k$. Thus the action of $\fg$  on $\lim\limits_{\ot} V/\overline{\fn_{\chi}^iV}$ is well defined. Finally, $V_{(N,\chi)}\cong \lim\limits_{\ot} V/\overline{\fn_{\chi}^iV}$ by Corollary \ref{cor:lim}.
\end{proof}
}

 \begin{lemma}\label{lem:locnil}
Let $N\subset L$ be nilpotent such that $N$ is normal in $L$ with commutative quotient. Let $V\in \Sc\mod(L)$. Suppose that $\fl$ acts locally nilpotently on $(V_N)^*$. Then $\fl$ acts locally nilpotently on $(V_{(N)})^*$.
\end{lemma}
\begin{proof}
By Lemma \ref{lem:GenFrobN}, it is enough to show that $\fl$ acts locally nilpotently on the space $(\Sc^*(N)^{(N)}\hot V^*)^N\cong \Sc^*(N)^{(N)}\hot (V_N)^*$. By Lemma \ref{lem:pol} this is isomorphic to $\C[N]\hot (V_N)^*$. Since $\fl$ acts locally nilpotently on $(V_N)^*$ by the assumption, and on $\C[N]$ by the nilpotency of $L$, the lemma follows.
\end{proof}

\begin{lemma}\label{lem:SvN}
Let $H=W\times Z$ be a Heisenberg group, and let $L,R\subset W$ be Lagrangian subspaces. Let $\chi$ be a non-trivial unitary character of $Z$. Then  $\Sc(H)_{(L\times Z,\chi)}\simeq \Sc(H)_{(R\times Z,\chi)}$ as representations of $H$, where $H$ acts on itself by left shifts, and $L$,$R$ and $Z$ act on $H$ by right shifts.
\end{lemma}
\begin{proof}
Let $L'\subset W$ be a Lagrangian subgroup complementary to both $L$ and $R$.
By Theorem \ref{thmB}, it is enough to establish an isomorphism between coinvariants of  $\Sc(H)_{(L\times Z,\chi)}$ and $\Sc(H)_{(R\times Z,\chi)}$
under $L'$. Since $L'\backslash H\cong L\times Z$, $(\Sc(H)_{(L\times Z,\chi)})_{L'}\cong \Sc(L\times Z)_{(L\times Z,\chi)}$. Similarly, $(\Sc(H)_{(R\times Z,\chi)})_{L'}\cong \Sc(R\times Z)_{(R\times Z,\chi)}$. Since the groups $L\times Z$ and $R\times Z$ are isomorphic, we have $\Sc(L\times Z)_{(L\times Z,\chi)}\cong \Sc(R\times Z)_{(R\times Z,\chi)}$. The lemma follows.
\end{proof}

\begin{prop}[{\cite[Proposition 3.0.1]{GGS2}}]\label{prop:LocNilp}
Let $P_2(\R)$ denote the group of affine transformations of the line. Let $N$ denote the unipotent
radical of $P_2(\R)$. Let $B_0$ denote the connected component
of the identity in $P_2(\R)$. Let $V_2\in \Rep^{\infty}(B_0)$. Suppose that $V$ is not generic, i.e. $V_{N,\psi}  = 0$ for
any non-trivial unitary character  $\psi$ of $N$. Then the  Lie algebra  of $N$ acts locally nilpotently on $V^*$.
\end{prop}

\begin{remark}
In the notation of Lemma \ref{lem:SvN},  $\Sc(H)_{(L\times Z,\chi)}$ and $\Sc(H)_{(R\times Z,\chi)}$ are not isomorphic as right $\fh$-modules. Indeed, $\fr$ acts locally nilpotently on the dual of $\Sc(H)_{(R\times Z,\chi)}$, but not on the dual of $\Sc(H)_{(L\times Z,\chi)}$. They will become isomorphic if we twist the right action of $\fh$ by an automorphism given by a symplectomorphism of $W$ that maps $L$ to $R$.
\end{remark}

\section{Pro-Whittaker models}\label{sec:Whit}
\DimaA{In this section we prove Theorem \ref{thm:Int_proWhit}. We will use the notation of \S \ref{subsec:IntproWHit}. }
Let $G$ be a real reductive group, and let $(S,\varphi)$ be a Whittaker pair for $G$.
Set $\fu:=\fg^S_{\geq 1}$ and let $\omega_{\varphi}$ denote the anti-symmetric form on $\fu$ given by $\omega_{\varphi}(X,Y):=\varphi([X,Y])$.
Choose a maximal isotropic subspace $\fl \subset \fu$, and let $U:=\Exp(\fu), \, L:=\Exp(\fl)$, and let $\chi$ be the character of $L$ defined by $\varphi$. For any $\pi\in \Sc\mod(G)$, define $\pi_{(S,\varphi)}:=\pi_{(L,\chi)}$.
Let $\pro\cW_{S,\varphi}^L:=\ind_U^G \Sc(U)_{(L,\chi)}\cong \Sc(G)_{(L,\chi)}$.

From the previous section we obtain:
\begin{lemma}\label{lem:SphiInd}
$\pi_{(S,\varphi)}\cong\pi\hot_G \pro\cW^L_{S,\varphi}\cong \pi \hot_G \ind_U^G  \Sc(U)_{(L,\chi)}$.
\end{lemma}

\begin{lemma}\label{lem:swap}
$\Sc(U)_{(L,\chi)}$ does not depend on the choice of $L$ as a \Fre space.
\end{lemma}
\begin{proof}

Let $\fv\subset \fu$ denote the radical of $\omega_{\varphi}$, and $V:=\Exp(\fv)$.  Let $K_{\chi}$ denote the kernel of $\chi$ on $V$. By Corollary \ref{cor:NUPol}, $$\Sc(U)_{(K_\chi)}\cong \C[[K_{\chi}]]\hot \Sc(U/K_{\chi}).$$
Now,
$$\Sc(U)_{(L,\chi)}\cong (\Sc(U)_{(K_\chi)})_{(L,\chi)}\cong (\C[[K_{\chi}]]\hot \Sc(U/K_{\chi}))_{(L,\chi)}\cong \C[[K_{\chi}]]\hot(\Sc(U/K_{\chi})_{(L,\chi)})$$
Since $U/K_{\chi}$ is a Heisenberg group, and $\Sc(U/K_{\chi})_{(L,\chi)}$
is the localization at $\chi$ of the generalized oscillator representation, it does not depend on the choice of $L$ by Lemma \ref{lem:SvN}.
\end{proof}
Thus we will sometimes drop the superindex $L$, and denote the pro-Whittaker model corresponding to $(S,\varphi)$ just by $\pro\cW_{S,\varphi}$. For a neutral pair $(h,\varphi)$, we will denote $\pro\cW_{h,\varphi}$ by $\pro\cW_{\varphi}$.
\Dima{As in \S \ref{subsec:IntproWHit}, we define $\WO(\pi)$ to be the set of all nilpotent orbits $O$ such that $\pi_{(\varphi)}\neq 0$ for any $\varphi\in O$, and by $\WS(\pi)$ the set of maximal elements of $\WS(\pi)$ under the closure order.
This definition is equivalent to the one used in \cite{GGS,GGS2} by the following lemma.
\begin{lem}
The following are equivalent
\begin{enumerate}[(i)]
\item $\pi_{(L,\chi)}=0$\label{proWhit0}
\item $\pi_{L,\chi}=0$\label{Whit0}
\item $\pi_{V,\chi}=0$\label{preWhit0}
\end{enumerate}
\end{lem}
\begin{proof}
The equivalence of \eqref{proWhit0} and \eqref{Whit0} follows from Lemma \ref{lem:locnil} by dualizing the spaces, since every nilpotent operator has a non-trivial kernel.

The equivalence of \eqref{Whit0} and \eqref{preWhit0} follows from Theorem \ref{thmB}.
\end{proof}

 }

\DimaA{Now we would like to prove Theorem \ref{thm:Int_proWhit}, in a way analogous to the proof of \cite[Theorem D]{GGS2}.
Let $(h,\varphi)$ be a neutral Whittaker pair with $[S,h]=0$. 
It exists by \cite[Lemma 3.0.2]{GGS}. Let $z:=S-h$, and for any $t\in [0,1]$ define $S_t=h+tz$. As in \cite[\S 3.2]{GGS}, we define a finite sequence of critical points $t_0=0<t_1<\dots<t_n=1$ by saying that $t$ is critical \Raul{if and only if} there exists $X\in \fg$ such that $[S_t,X]=X$ and $[z,X]\neq 0$.
For each $i$ define $\fu_i:=\fg^{S_{t_i}}_{\geq 1}$. 

Consider the space $\{X\in \fg \mid  [z,X]=0, [S,X]=X\}$. Then the form $\omega_{\varphi}$ restricts to a sympletic form on this space, and we choose a Lagrangian subspace $\fm$.

For every $i$ we define two nilpotent subalgebras $\fl_i,\fr_i\subset \fg$ by 
 \begin{equation}\label{=lt}
 \fl_i:=\fm+(\fu_{t_i}\cap \fg^{Z}_{< 0})+\Ker(\omega_{\varphi}|_{\fu_{t_i}})\text{ and }\fr_t:=\fm+(\fu_t\cap \fg^{Z}_{> 0})+\Ker(\omega_{\varphi}|_{\fu_{t_i}}).
\end{equation}

By \cite[Lemma 3.2.7]{GGS}, the $\fl_i$ and $\fr_i$ are maximal isotropic subspaces of $\fu_i=\fg^{S_{t_i}}_{\geq 1}$, and we have $r_{i}\subset l_{i+1}$ for every $i<n$. Thus we have a natural map 
$\pro\cW_{S_i,\varphi}^{R_{i}}\onto \pro\cW_{S_{i+1},\varphi}^{L_{i+1}}$.

Lemma \ref{lem:swap} gives a sequence of epimorphisms
\begin{equation}\label{=modSeq}
\pro\cW_{\varphi}=\pro\cW_{S_0,\varphi}^{R_0}\onto \pro\cW_{S_1,\varphi}^{L_1}\cong \pro\cW_{S_1,\varphi}^{R_1}\onto \dots \onto \pro\cW_{S_n,\varphi}^{L_n}=\pro\cW_{S,\varphi}.
\end{equation}
By tensoring over $G$ with a representation $\pi\in \Rep^{\infty}(G)$, the sequence \eqref{=modSeq} gives us the sequence of maps
\begin{equation}\label{=modSeqpi}
\pi_{(\varphi)}=\pi_{(R_0,\chi_\varphi)}\to \pi_{(L_1,\chi_\varphi)}\cong \pi_{(R_1,\chi_\varphi)}\to \dots \to \pi_{(L_n,\chi_\varphi)}=\pi_{(S,\varphi)}.
\end{equation}

\begin{thm}\label{thm:proWhit}
Let $\pi\in \Rep^{\infty}(G)$, and let $(S,\varphi)$ be a Whittaker pair with $\Dima{G\cdot}\varphi\in \WS(\pi)$. Then the map $\pi_{(\varphi)}\to \pi_{(S,\varphi)}$ given by \eqref{=modSeqpi} is an \DimaA{isomorphism}.
\end{thm}
\begin{proof}

We have to show that for every $i<n$, the natural map $\pi_{(R_i,\chi_{\varphi})}\to \pi_{(L_{i+1},\chi_{\varphi})}$ is an isomorphism. For this, it is enough to show that the dual map  is an isomorphism, {\it i.e.}, that every generalized $(R_i,\chi_{\varphi})$-equivariant functional on $\pi$ is, automatically, a generalized $(L_{i+1},\chi_{\varphi})$-equivariant \Raul{functional too}. For this, it is enough to show that  the Lie algebra $\fl_{i+1}/\fr_i$ acts locally nilpotently on $(\pi^*)^{(\fr_i,\chi_{\varphi})}$. \Raul{But then it suffices} to show that for any $\pi\in \Rep^{\infty}(G)$ with $\varphi\in \WS(\pi)$, the Lie algebra $\fl_{i+1}/\fr_i$ acts locally nilpotently on the space of equivariant functionals $(\pi^*)^{\fr_i,\chi_{\varphi}}$.

This is shown in \cite[\S 4.2]{GGS2}, using the condition $\varphi\in \WS(\pi)$, Proposition \ref{prop:LocNilp}, and the notion of quasi-Whittaker coefficient defined in {\it loc. cit.}. More precisely, this follows from \cite[Lemma 4.9, Proposition 4.12 and Corollary 4.20]{GGS2}.
\end{proof}}



\begin{remark}
\begin{enumerate}[(i)]
\item
By Lemma \ref{lem:Mod}, $\pi_{(\varphi)}$ and $\pi_{(S,\varphi)}$ have natural structures of $\fg$-modules. However, they are in general not isomorphic as $\fg$-modules unless $\fr_0\subset \fr_n$. Indeed, the dual $\pi_{\varphi}^*$ has locally finite action of $\fr_0$, \Raul{whereas} $\pi_{(S,\varphi)}^*$
has locally finite action of $\fr_n$.
\item It is easy to see that if we drop the condition  $\varphi\in \WS(\pi)$, the map  $\pi_{(\varphi)}\to \pi_{(S,\varphi)}$ given by \eqref{=modSeqpi} will still be onto, since every map in the sequence is either an isomorphism or a quotient.

However, it will not be a monomorphism in general. For example, let $G=\GL_2(\R)$, $\varphi = 0$ and let  $S$ define the unipotent radical $N$ of a Borel subgroup.
Then for every $\pi$, we have $\pi_{(\varphi)}=\pi$ and $\pi_{(S,\varphi)}=\pi_{(N)}$.
Let $\bar B$ denote the opposite Borel, and let $\pi$ be the principal series representation $\Ind_{\bar B}^G1$.
Then we can view $\pi$ as the  space of infinitely smooth functions on $N$ \DimaA{that are bounded at infinity as well as all their derivatives}. The group $N$ acts on this space by shifts.
Thus the kernel of the map $\pi \to \pi_{(N)}$ includes the space of Schwartz functions that for any $i$ are $i$-th derivatives of Schwartz functions. These are the Fourier transforms of Schwartz functions that are flat at zero.
Note that this kernel does not intersect the Harish-Chandra module of $\pi$.
\end{enumerate}
\end{remark}

\appendix

\section{Generalized Kashiwara's equivalence}\label{app:Kash}

Let $k$ be a field of characteristic zero. Let $(W,\omega)$ be a symplectic vector space over $k$, and let $\fh:=W\oplus k$ be the corresponding Heisenberg Lie algebra. Denote by $Z\cong k$ the center of $\fh$, and by $z$ the element of $Z$ corresponding to $1\in k$.
Let $\fl \subset W$ be a Lagrangian subspace.
\Dima{
As in \S\ref{subsec:A}, we denote by
$\cM_{\fl}(\fh)$  the category of $\fh$-modules on which $\fl$ acts locally nilpotently, by $\cM_{\fl}(\fh)^{\times}$ the subcategory of modules on which $z$ acts invertibly, by $\cM(Z)$ the category of $Z$-modules, and by $\cM(Z)^{\times}$ the subcategory of modules on which $z$ acts invertibly.

In \S\ref{subsec:A} we defined the following pair of functors between these categories:
\[   F:  \cM_{\fl}(\fh)^{\times}\leftrightarrows \cM(\fz)^{\times}: G,\quad  F(M):=M^\fl,\; G(N):=\cU(\fh)\otimes_{\cU(\fp)}E(N). \]

For any $M\in \cM(\fh)$ one has a natural map $\cU(\fh)\otimes_{\cU(\fh \oplus Z)}M\to M$, given by the action map $\cU(\fh)\times M\ \to M$. Restricting to the subspace $F(M)\subset M$, we get a natural map $\alp:G(F(M))\to M$.

Theorem \ref{thmA} says that the functors $F$ and $G$ are quasi-inverses of each other, and thus establish an equivalence of categories.
To prove this we will need
 the following proposition, that we will prove in the next subsection.
\begin{prop}\label{prop:Kash}
For any $M\in \cM_{\fl}(\fh)^{\times}$, the map $\alp$ is an isomorphism.
\end{prop}

\begin{proof}[Proof of Theorem \ref{thmA}]
By Proposition \ref{prop:Kash}, $\alp$ is an isomorphism $G\circ F\simeq \Id$. For the isomorphism $F\circ G\simeq \Id$, let $N\in \cM(\fz)^{\times}$ and note that  $\fl$ acts locally nilpotently on $G(N)$.  Choose a  Lagrangian $\fr \subset W$ complementary to $\fl$. Note that we have a natural vector space isomorphism
\begin{equation}\label{=TS}
G(N)\cong S(\fr)\otimes_k N \,,
\end{equation} where $S$ denotes the symmetric algebra. This isomorphism is defined by decomposing $\cU(\fh)\cong S(\fr)\otimes_k \cU(\fl\oplus Z)$, and applying $\cU(\fl \oplus Z)$ to $N$ (where $\fl$ acts trivially). Here, we use that $\fl \oplus Z$ is a commutative normal subalgebra of $\fh$. Using \eqref{=TS} we get a natural isomorphism $F(G(N))\cong N$.
\end{proof}

\begin{remark}
Let $\mathfrak{i}\subset \fh$ denote the ideal generated by $z-1$. Then
the universal enveloping algebra $\cU(\fh/\mathfrak{i})$ is isomorphic to the Weyl algebra of $\fl$, {\it i.e.} the algebra of differential operators on $\fl$\ with polynomial coefficients. Thus, restricting Theorem \ref{thmA} to the subcategories on which $z$ acts by 1 we obtain the  affine spaces special case of the famous Kashiwara's equivalence from the theory of D-modules. The Kashiwara's equivalence is a natural equivalence between the category of D-modules on a closed submanifold and  the category of D-modules defined on the ambient space, but supported on the submanifold.
 Our proof also mimics the proof of Kashiwara's equivalence (see {\it e.g.} \cite[Theorem 1.6.1]{HTT}).
\end{remark}
}
\subsection{Proof of Proposition \ref{prop:Kash}}

 We will prove the proposition by induction on $\dim \fl$.
\DimaA{Choose a  Lagrangian $\fr \subset W$ complementary to $\fl$.} Let $x\in \fl$ and $y\in \fr$ such that $[x,y]=z$. Let $M\in \cM_{\fl}(\fh)^{\times}$. For any $i\geq 0$, define two vector subspaces of $M$:
\begin{equation}\label{=Filt}
M_i:=\Ker x^{i+1}, \quad \quad N_i:=y^i \Ker x
\end{equation}

\begin{lemma}[Key lemma]\label{lem:key}For any $i\geq 0$ we have
$M_i=\bigoplus_{j=0}^i N_j$

\end{lemma}

To prove this lemma we will use the following one.

\begin{lemma}\label{lem:filt}
Define $M_{-1}:=N_{-1}:=0$. Then for any $i\geq 0$ we have
\begin{enumerate}[(i)]
\item \label{help:Gp}$(yx-iz)(N_i)=0$.
\item \label{help:x} $xN_{i}\subset N_{i-1}$.
\item \label{help:GM}$N_i\subset M_{i}$.
\item \label{help:MGp}$(yx-iz)(M_{i})\subset M_{i-1}$.
\end{enumerate}
\end{lemma}
\begin{proof} For $i=0$ this holds by definition. For $i>0$ we use the identity $[x,y^i]=izy^{i-1}$.
\begin{enumerate}[(i)]
\item $(yx-iz)y^i=y^{i+1}x$, which vanishes on $\Ker x$.
\item $xy^{i}\Ker x=[x,y^i](\Ker x)=izy^{i-1}(\Ker x)=y^{i-1}(\Ker x)=N_{i-1}$.
\item Follows from  \eqref{help:x} by induction on $i$.
\item $x(yx-iz)=-(i-1)xz+yx^2=(yx-(i-1)z)x$. The statement follows now by induction on $i$.
\end{enumerate}
\end{proof}

\begin{proof}[Proof of Key Lemma \ref{lem:key}]
By Lemma \ref{lem:filt}\eqref{help:Gp}, $N_i$ lies in the $i$-eigenspace of the operator $z^{-1}yx$, thus the sum of all $N_i$ is direct. By Lemma \ref{lem:filt}\eqref{help:GM}, $\bigoplus_{j=0}^i N_j\subset M_i$.
Let us now prove the converse inclusion by induction. The base $i=0$ holds by definition. For the induction step let $i> 0$. Then by Lemma \ref{lem:filt}\eqref{help:MGp} and the induction hypothesis we have
\begin{equation}\label{=key}
(yx-iz)M_{i}\subset M_{i-1}=\bigoplus_{j=0}^{i-1} N_j.
\end{equation}
Since $xM_{i}\subset M_{i-1}$, and $yN_{j}=N_{j+1}$, the induction hypothesis implies  $yxM_i\subset  \bigoplus_{j=0}^{i}N_j.$ From this and \eqref{=key} we obtain $izM_i\subset \bigoplus_{j=0}^{i}N_j.$
Since $iz$ is invertible and commutes with $x$ and $y$, $(iz)^{-1}$ preserves $N_j$. Thus $M_i\subset \bigoplus_{j=0}^{i}N_j.$
\end{proof}

\DimaA{\begin{cor}\label{cor:key}
\begin{enumerate}[(i)]
\item $M=\oplus N_i$. \label{it:dsum}
\item $y$ has no kernel on $M$. \label{it:1to1}
\end{enumerate}
\end{cor}
\begin{proof}
\eqref{it:dsum} follows from Lemma \ref{lem:key}. For \eqref{it:1to1} it is enough to show that $y$ has no kernel on $N^i$ for any $i$. Let $m\in N^i$ such that $ym=0$. Then by Lemma \ref{lem:filt}\eqref{help:Gp} we have $$yxm=izm=i(xy-yx)m=-iyxm \Rightarrow (i+1)yxm=0 \Rightarrow yxm=0\Rightarrow zm=0\Rightarrow m=0.$$
\end{proof}
}
\begin{proof}[Proof of Proposition \ref{prop:Kash}]
For any $i\geq 0$ define $M^{(i)}:=S^i(\fl)F(M)$.
We claim that
\begin{equation}
M=\bigoplus_{i\geq 0}M^{(i)}
\end{equation}
and that for any $i$, the linear epimorphism
\begin{equation}
\alp_i: S^i(\fr)\otimes_k F(M)\to M^{(i)}
\end{equation}
given by $\alp$ is an isomorphism.
These two statements imply the proposition.

Both statements follow from \DimaA{Corollary \ref{cor:key}}  by induction on $n=\dim \fl$. For this, we choose a basis $x_1,\dots x_n$ for $\fl$, and let $y_1,\dots y_n\in \fr$ be the dual basis. Then for any $1\leq j \leq n$ we let $\fh_j$ be the subalgebra of $\fh$ generated by $x_1,\dots, x_j,y_1,\dots y_j,z$, and $N_j$ be the $\fh_j$-submodule of $M$ generated by $F(M)$. Then the two statements for $N_j$ follow by induction on $j$ from \DimaA{Corollary \ref{cor:key}}.
\end{proof}
\begin{remark}
The proposition does not hold for infinite-dimensional $\fh$. Indeed, let $W$ be of countable dimension, and let $I$ be the left ideal of $\cU(\fh)$ generated by $\fl$ and by $z-1$. Let $M$ be the direct product of infinitely many copies of $\cU(\fh)/I$. Then, the proposition does not hold for $M$.

We wonder whether the proposition  holds for infinite-dimensional $\fh$, but finitely generated $M$.
\end{remark}

\end{document}